\documentclass[12pt]{amsart}
\usepackage{amsmath,amsthm,amssymb,mathrsfs,amscd,color,bbm}
\usepackage{url}
\usepackage{graphicx}
\usepackage{pb-diagram} 
\usepackage[all]{xy}
\usepackage{comment}
\usepackage{mathtools}
\usepackage{centernot}
\usepackage{placeins} % for \FloatBarrier 

\makeatletter
\newcommand{\tpitchfork}{%
  \vbox{
    \baselineskip\z@skip
    \lineskip-.52ex
    \lineskiplimit\maxdimen
    \m@th
    \ialign{##\crcr\hidewidth\smash{$-$}\hidewidth\crcr$\pitchfork$\crcr}
  }%
  
}
\makeatother

\newtheorem{prop}{Proposition}[section]
\newtheorem{lemma}[prop]{Lemma}

\newtheorem{cor}[prop]{Corollary}
\newtheorem{conjecture}[prop]{Conjecture}
\newtheorem{question}[prop]{Question}
\newtheorem{theorem}[prop]{Theorem}

\theoremstyle{remark}
\newtheorem{remark}[prop]{Remark}
\numberwithin{equation}{section}

\author{Yoshitaka Saiki, 
Hiroki Takahasi,
Kenichiro Yamamoto,
James A. Yorke
}

\address{Graduate School of Business Administration, Hitotsubashi University, Tokyo, 186-8601, JAPAN} 
\email{yoshi.saiki@r.hit-u.ac.jp}

 \address{ Keio Institute of Pure and Applied Sciences, Department of Mathematics,
 Keio University, Yokohama,
 223-8522, JAPAN} 
 \email{hiroki@math.keio.ac.jp}

\address{Department of General Education, Nagaoka University of Technology, Nagaoka 940-2188, JAPAN}
\email{k\_yamamoto@vos.nagaokaut.ac.jp}

\address{Institute for Physical Science and Technology and the Departments of Mathematics and Physics, University of Maryland College Park, MD 20742, USA} 
\email{yorke@umd.edu}

\date{}
\thanks{}

\begin{document}
\title[The dynamics of the heterochaos baker maps]{The dynamics of
the heterochaos baker maps
}

%{\textbf{\today}}
\begin{abstract}
The heterochaos baker maps are piecewise affine maps of the unit square or cube introduced in [Nonlinearity {\bf 34}, 2021, 5744--5761], 
to provide a hands-on, elementary understanding of complicated phenomena in systems of large degrees of freedom. 
We review recent progress on a dynamical systems theory of the heterochaos baker maps, and present new results on properties of measures of maximal entropy and the underlying Lebesgue measure. 
We address several conjectures and questions that may illuminate new aspects of heterochaos and inspire future research.
\end{abstract}

\maketitle

\tableofcontents
\newpage
\section{Introduction}
\label{sec:introduction}

A natural approach to understanding complicated systems is to analyze 
a collection of 
simple %examples 
models that retain some essential features of the complexity of the original systems.
Among a number of %simple 
examples 
proposed so far as %simple
models of chaotic dynamics, the quadratic map \cite{May76} and the H\'enon map \cite{He76} played prominent roles 
in advancing our knowledge on chaos beyond uniform hyperbolicity. Thanks to their algebraically simple definitions, numerical experiments can be done, which hint good theorems, and in turn, rigorous results inspire new numerical experiments %which may lead 
leading to further discoveries of new phenomena.
We would like to have %expect 
such a %(partially at least) 
%successful, 
fruitful 
model %\textcolor{red}{(remove) in higher dimension} 
%for other interesting phenomena
on which rigorous analysis and numerical experiments work together nicely.

 For multidimensional systems that are not uniformly hyperbolic, %it often happens that
 the dynamics in some regions can be unstable in 
 more directions than %that 
 in other regions. This suggests that the dimension
 of unstable directions
 (i.e., the number of `positive Lyapunov exponents' counted with multiplicity)
 is not constant and varies from point to point.
 This situation is called
 {\it the unstable dimension variability}.
%\cite{DGSY94,KKGOY}. 
Certainly, the unstable dimension variability is  
universal  
in systems of large degrees of freedom, like weather models, and it is an intrinsic hurdle for understanding their global dynamics.

Due to the unstable dimension variability, different periodic orbits can have different unstable dimensions and they can be embedded densely in the same chaotic invariant sets.
For non-hyperbolic diffeomorphisms, 
this type of coexistence phenomenon was recognized 
 since the 60s 
\cite{AS70,mane_1978,shub_1971,Si72}. 
%These theoretical results do not aim to specific systems, but aim to generic diffeomorphisms.
% {\color{red} (remove) These theoretical results aim at constructing a theory on generic diffeomorphisms.  Bonatti and D\'iaz~\cite{bonatti_1996} studied heterodimensional cycles that provide a mechanism to create the unstable dimension variability and proposed a structure called the blenders to show that the unstable dimension variability can persist under perturbations.}
A unified mechanism (heterodimensional cycle and blender) for the robustness of the unstable dimension variability under perturbations was introduced in \cite{bonatti_1996,D95}. These theoretical results do not aim to specific systems, but aim to generic diffeomorphisms.
Since the 90s, various numerical results 
%approaches 
have been obtained %developed
which %numerically 
detect periodic orbits of different unstable dimensions in various specific systems (see e.g.,
\cite{dawson_1996,DGSY94,gritsun_2008,gritsun_2013,KKGOY,saiki_2018c}). 
 These and other numerical results have created among scientists, including applied mathematicians, a strong need for a simple, computably tractable model that displays
 the unstable dimension variability and associated interesting phenomena.

In \cite{STY21}, the first, second and the last authors 
 introduced {\it heterochaos baker maps},
%\textcolor{red}{(remove) as simple models to understand the unstable dimension variability with minimal pre-requisites.}
 as minimal models to understand the unstable dimension variability.
They are piecewise affine maps of the unit square $[0,1]^2$ or the cube $[0,1]^3$, reminiscences of the baker map \cite{H56}.
In \cite{STY21} they proved that periodic points of these maps with different unstable dimensions are densely embedded in the phase spaces. Part of these periodic points were numerically detected, and the beautiful plots of them in \cite[Figure~5]{STY21} hint an underlying structure.
On the heterochaos baker maps we are witnessing a harmony %synthesis %marriage 
of rigorous analysis and numerical works.
The term `heterochaos' was coined in \cite{STY21} 
that roughly means the coexistence of multiple, qualitatively different chaotic behaviors.

Since the appearance of the paper \cite{STY21}, interesting developments are subsequently taking place, see e.g., \cite{HK22,STY22,T23,TY23}.
In this paper we review recent progress on a dynamical systems theory of the heterochaos baker maps, and present new results.
We address several conjectures and questions that may illuminate new aspects of heterochaos and inspire future research.

\subsection{The heterochaos baker maps}\label{hetero-def}
Let $M\geq2$ be an integer, which is fixed throughout.
 For each $a\in (0,\frac{1}{M})$
 define a map $\tau_a\colon[0,1]\to[0,1]$ by 
\[\tau_a(x_u)=\begin{cases}\vspace{1mm}
\displaystyle{\frac{x_u-(k-1)a}{a}}&\text{ on }[(k-1)a,ka),\ k=1,\ldots,M,\\ \displaystyle{\frac{x_u-Ma}{1-Ma}}&\text{ on }
[Ma,1].
\end{cases}\]
The heterochaos baker maps $(f_{a})$, $(f_{a,b})$ are skew product maps over the piecewise affine fully branched map $\tau_a$.
We introduce a set of $2M$ symbols
 \begin{equation}\label{D}D=\{\alpha_1,\ldots,\alpha_M\}\cup\{\beta_1,\ldots,\beta_M\}.\end{equation}
For each symbol $\gamma\in D$, we define a rectangular domain $\Omega_\gamma^+$ in $[0,1]^2$ by
\[\Omega_{\alpha_k}^+=\left[(k-1)a,ka\right)\times
\left[0,1\right]\text{ for }k=1,\ldots,M, 
\]
and
\[\Omega_{\beta_k}^+=\begin{cases}\vspace{1mm}
\left[Ma,1\right]\times
\displaystyle{\left[\frac{k-1}{M },\frac{k }{M }\right)}&
\text{ for }k=1,\ldots,M-1,\\
\left[Ma,1\right]\times
\displaystyle{\left[\frac{M-1 }{M},1\right]}&\text{ for }k=M.
\end{cases}\]
Note that these domains are pairwise disjoint and
cover the entire square $[0,1]^2$.
We now
define $f_a\colon[0,1]^2\to[0,1]^2$  by
\[\begin{split}
  f_a(x_u,x_c)=
  \begin{cases}
    \displaystyle{\left(\tau_a(x_u),\frac{x_c}{M}+\frac{k-1}{M}\right)}&\text{ on }\Omega_{\alpha_k}^+,\ k=1,\ldots,M,\\
   \displaystyle{\left (\tau_a(x_u),Mx_c-k+1\right)}&\text{ on }\Omega_{\beta_k}^+,\ k=1,\ldots,M.
   \end{cases}
\end{split}\]
See \textsc{Figure}~\ref{hetero2} for the case $M=2$.

Put
$\Omega_\gamma=\Omega_\gamma^+\times
\left[0,1\right]$ for each $\gamma\in D$, and set
$\Delta=\left(0,\frac{1}{M}\right)^2.$
For each $(a,b)\in \Delta$
we define $f_{a,b}\colon [0,1]^3\to[0,1]^3$ by 
\[\begin{split}
  f_{a,b}(x_u,x_c,x_s)=
  \begin{cases}
    (f_{a}(x_u,x_c),
    (1-Mb)x_s)&\text{ on }\Omega_{\alpha_k},\ 
    k=1,\ldots,M,\\
   \displaystyle{\left (f_{a}(x_u,x_c),bx_s+1+b(k-M-1)\right)}&\text{ on }\Omega_{\beta_k},\ k=1,\ldots,M.
   \end{cases}
\end{split}\]
See \textsc{Figure}~\ref{hetero3} for the case $M=2$. Clearly $f_a$ is the projection of $ f_{a,b}$ to the $(x_u,x_c)$-plane.
 Note that $f_{a,b}$ is one-to-one except on the points in the boundaries of the domains $\Omega_\gamma$ where it is discontinuous and at most three-to-one. 
If there is no fear of confusion, we will often omit $a$, $b$ and write $f$ for simplicity.

\begin{figure}
\begin{center}
\includegraphics[height=4.4cm,width=12cm]{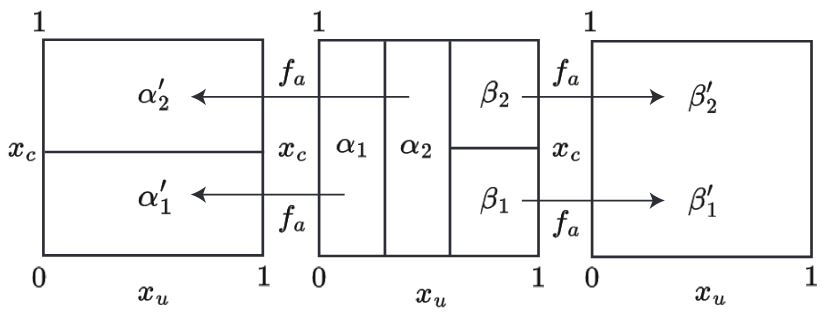}
%{heterochaos2d.eps}
\caption
{The heterochaos baker map $f_{a}$ with $M=2$. For each $\gamma\in D$,
the domain $\Omega_\gamma^+$  and its image are labeled with $\gamma$ and $\gamma'$ respectively: $f_a(\Omega_{\beta_1}^+)=[0,1]\times[0,1)$ and $f_a(\Omega_{\beta_2}^+)=[0,1]^2$.}\label{hetero2}
\end{center}
\end{figure}
 
 \begin{figure}[b]
\begin{center}
\includegraphics[height=4.8cm,width=13cm]{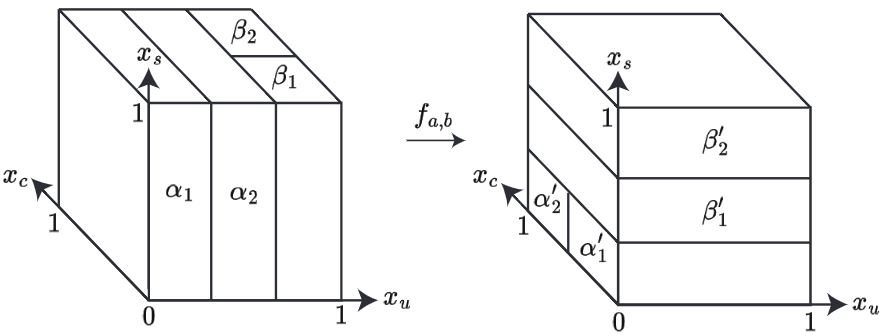}
%{heterochaos3d-revised.eps}
\caption
{The heterochaos baker map $f_{a,b}$ with $M=2$. For each $\gamma\in D$,
the domain $\Omega_\gamma$  and its image are labeled with $\gamma$ and $\gamma'$ respectively. }\label{hetero3}
\end{center}
\end{figure}

The heterochaos baker maps $f_{\frac{1}{3}}$ and
$f_{\frac{1}{3},\frac{1}{6}}$ with $M=2$ 
were introduced in \cite{STY21}, and the maps of the above general form were introduced in \cite{TY23}.
  In \cite{STY22}, some variants of $f_{a,b}$ were considered in the context of fractal geometry and homoclinic bifurcations of three dimensional diffeomorphisms.  Iterated function systems intimately related to $f_{a,b}$ was considered in \cite{HK22}.
 Earlier than \cite{STY21}, 
 a simple robust model of partially hyperbolic systems called `a porcupine-like horseshoe' in dimension three was introduced and investigated in
  \cite{DG12}, and its dynamical properties were further analyzed in \cite{DGR11,DGR14}. Owing to the piecewise affinity, the heterochaos baker maps are not robust under perturbations of systems.

\subsection{Overview of the paper}
 The rigorous theory of the heterochaos baker maps has begun 
 in \cite{STY21} with the analysis of their periodic points.
  Symbolic dynamics and measures of maximal entropy were investigated in \cite{TY23}. 
 Rates of mixing with respect to the Lebesgue measure were explored in \cite{T23}.
In Section~2 we review these 
results, together with numerical ones, and state main results of this paper that are summarized as follows: 

\begin{itemize}
\item the existence of precisely two ergodic measures of maximal entropy (MMEs) for all parameters (Theorems~\ref{thm-a}), the connection between the two ergodic MMEs and the Lebesgue measure
(Theorem~\ref{thm-a}). 

\item the continuous dependence of the two ergodic MMEs on parameters (Theorem~\ref{thm-c}),

\item the Bernoulliness of the Lebesgue measure (Theorem~ \ref{Bernoulli-thm}),

\end{itemize}

As inferred from their definitions, the heterochaos baker maps have natural symbolic representations.
 In \cite{TY23}, an interesting connection was found between the heterochaos baker maps and the Dyck shift \cite{Kri74}, which implies that {\it one is a mirror image of the other}. This connection underlies proofs of all the main results of this paper. In Section~3 we start with the definition of the Dyck shift, and state
 its connection with the heterochaos baker maps (Theorem~\ref{TY-thm}).
 We then recall basic structures of the Dyck shift
  from \cite{Kri74}, and that of the heterochaos baker maps from \cite{TY23}.
  Using these ingredients, in Section~4 we prove the main results. In Appendices A and B we discuss some variants of the heterochaos baker maps.

%The last section has an independent character. Since the heterochaos baker maps are piecewise affine, they are not robust under perturbations. (piecewise) smooth systems may be models of real world phenomena in time evolution, and real significance may be accorded only to those properties which are robust. we discuss the heterochaos horseshoe maps considered in \cite{STY22}.

\section{The dynamics}
In this section we review existing
results on the dynamics of the heterochaos baker maps, 
along with numerical results, and state the main results of this paper. In addition, we address several conjectures and questions that may illuminate new aspects of heterochaos and inspire future researches.

\subsection{Periodic points of different unstable dimensions}
The heterochaos baker maps are one of the simplest non-hyperbolic systems.
Under the forward iteration of $f=f_{a,b}$,
the $x_u$-direction is expanding by factor $\frac{1}{a}$ or $\frac{1}{1-Ma}$ and the $x_s$-direction is contracting by factor $1-Mb$ or $b$. The $x_c$-direction is a center: contracting on $\bigcup_{k=1}^M\Omega_{\beta_k}$ and expanding on $\bigcup_{k=1}^M\Omega_{\beta_k}$.
%where
%\[\Omega_\alpha=\bigcup_{k=1}^M\Omega_{\alpha_k}\ \text{ and }\ \Omega_\beta=\bigcup_{k=1}^M\Omega_{\beta_k}.\]
The local stability in the $x_c$-direction along each orbit  
is determined by the asymptotic time average of 
the function $\phi^c\colon [0,1]^3\to\mathbb R$ given by
  \begin{equation}\label{geometric-c}\phi^c(x)=\begin{cases}-\log M&\text{ on }  \bigcup_{k=1}^M\Omega_{\alpha_k},\\
 \log M&\text{ on }  \bigcup_{k=1}^M\Omega_{\beta_k}.\end{cases}\end{equation}

%Periodic points are the skelton of chaotic dynamical systems \cite{Li_1975}.
Let ${\rm int}(\cdot)$ denote the interior operation in $\mathbb R^3$ or $\mathbb R^2$. For each
$n\in\mathbb N$,
elements of the set 
\[{\rm Per}_n(f)=
\left\{x\in\bigcap_{i=0}^{n-1}f^{-i}\left(\bigcup_{\gamma\in D}{\rm int}(\Omega_\gamma)\right)\colon f^n(x)=x\right\}\] 
 are called {\it periodic points of period $n$}.
  For each $x\in{\rm Per}_n(f)$, the average value $\chi^c(x)=\frac{1}{n}\sum_{i=0}^{n-1}\phi^c(f^i(x))$ equals the exponential growth rate of the second diagonal element of the Jacobian matrix of $f^n$ at $x$.
The unstable dimension of $x$
is either $1$ or $2$ according as  $\chi^c(x)$ is negative or positive. 
We decompose ${\rm Per}_n(f)$ into three subsets:
\[\begin{split}
{\rm Per}_{\alpha,n}(f)&=\{x\in{\rm Per}_n(f)\colon \chi^c(x)<0\};\\
{\rm Per}_{\beta,n}(f)&=\{x\in {\rm Per}_n(f)\colon \chi^c(x)>0\};\\
{\rm Per}_{0,n}(f)&=\{x\in{\rm Per}_n(f)\colon \chi^c(x)=0\}.
\end{split}\] 
Periodic points in 
the first and the second sets are referred to as {\it $1$-unstable} and {\it $2$-unstable} respectively.
The first and the second sets are finite, while the last one may not be finite: it may contain
segments parallel to the $x_c$-axis.
We exclude from further consideration
periodic points that are contained in the last set.
The notion of $1$-unstable, $2$-unstable periodic points carries over to $f_a$ and other piecewise differentiable maps in an obvious way.

The next theorem shows a manifestation of the unstable dimension variability in periodic points, i.e., {\it a heterochaos on topological level}. Although only the special case $(a,b)=(\frac{1}{3},\frac{1}{6})$ with $M=2$ was treated in \cite{STY21}, essentially the same proof works to cover general cases.

\begin{figure}[b]
%  \begin{minipage}[b]{0.32\linewidth}
% {\includegraphics[width=0.9\textwidth]{graph-upos-sad-p11-ver2s.eps}}
% {\centering \small period 11, 1-unstable}
% \end{minipage}
% \begin{minipage}[b]{0.32\linewidth}
% {\includegraphics[width=0.9\textwidth]{graph-upos-sad-p12-ver2s.eps}}
% {\centering \small period 12, 1-unstable}
% \end{minipage}
% \begin{minipage}[b]{0.32\linewidth}
% {\includegraphics[width=0.9\textwidth]{graph-upos-sad-p13-ver2s.eps}}
% {\centering \small period 13, 1-unstable}
% \end{minipage}
% \begin{minipage}[b]{0.32\linewidth}
% {\includegraphics[width=0.9\textwidth]{graph-upos-rep-p11-ver2s.eps}}
% {\centering \small period 11, 2-unstable}
% \end{minipage}
%  \begin{minipage}[b]{0.32\linewidth}
% {\includegraphics[width=0.9\textwidth]{graph-upos-rep-p12-ver2s.eps}}
% {\centering \small period 12, 2-unstable}
% \end{minipage}
% \begin{minipage}[b]{0.32\linewidth}
% {\includegraphics[width=0.9\textwidth]{graph-upos-rep-p13-ver2s.eps}}
% {\centering \small period 13, 2-unstable}
% \end{minipage}
% \begin{minipage}[b]{0.32\linewidth}
% {\includegraphics[width=0.9\textwidth]{graph-upos-all-p11-ver3s.eps}}
% {\centering \small period 11, $1\&2$-unstable}
% \end{minipage}
%  \begin{minipage}[b]{0.32\linewidth}
% {\includegraphics[width=0.9\textwidth]{graph-upos-all-p12-ver3s.eps}}
% {\centering \small period 12, $1\&2$-unstable}
% \end{minipage}
% \begin{minipage}[b]{0.32\linewidth}
% {\includegraphics[width=0.9\textwidth]{graph-upos-all-p13-ver3s.eps}}
% {\centering \small period 13, $1\&2$-unstable}
% \end{minipage}
\includegraphics[width=0.9\textwidth]{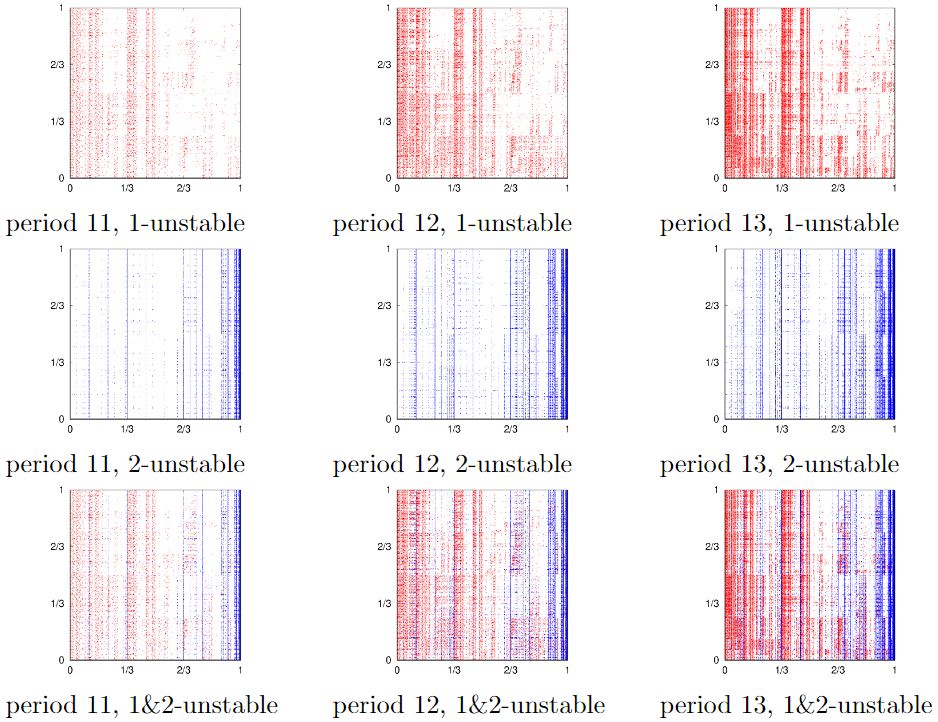}
%{upo-p11-13-1u-2u-12u.eps}
\caption{Part of periodic points of $f_{\frac{1}{3}}$ with $M=2$.
The first row shows $1$-unstable periodic points, the second row shows $2$-unstable periodic points, and
the third row shows both of the $1$- and $2$-unstable periodic points.
}
\label{fig:periodicorbits}
\end{figure}

\begin{theorem}[\cite{STY21}, Theorem~1.1]\label{thm-1}
For any $(a,b)\in \Delta$ the following statements hold:
\begin{itemize}
\item[(a)] $f_{a,b}$ has a dense orbit in $[0,1]^3$ and $f_a$ has a dense orbit in $[0,1]^2$;
\item[(b)] The set of $1$-unstable periodic points of $f_{a,b}$ is dense in $[0,1]^3$, and
the set of $2$-unstable periodic points of $f_{a,b}$ is dense in $[0,1]^3$;

\item[(c)] The set of $1$-unstable periodic points of $f_{a}$ is dense in $[0,1]^2$, and
the set of $2$-unstable periodic points of $f_a$ is dense in $[0,1]^2$.
\end{itemize}
\end{theorem}

 Item (a) asserts the topological transitivity of the maps.
 \textsc{Figure}~\ref{fig:periodicorbits} shows partial plots of 1-unstable, 2-unstable periodic points of $f_{\frac{1}{3}}$ with $M=2$ of period $11$, $12$, $13$.
 Items (b) and (c) indicate that
 the white spots are actually filled with periodic points of higher periods.

The coexistence phenomenon of periodic points 
with different unstable dimensions in the same transitive set as in Theorem~\ref{thm-1} 
should be widespread among systems that are not uniformly hyperbolic.
Indeed, this phenomenon (and even its robustness under perturbations of the systems) for diffeomorphisms was already known since the 60s 
\cite{AS70,bonatti_1996,mane_1978,shub_1971,Si72}. However, these results do not immediately apply to specific systems of interest. Consequently, we only have a handful of examples for which the coexistence phenomenon has actually been verified.
For example,
Pikovsky and Grassberger
\cite{pikovsky_1991} introduced  {\it a coupled tent map}
\[\begin{split}x_{n+1}&=(1-\omega)h_a(x_n)+\omega h_a(y_n)\\
y_{n+1}&=\omega h_a(x_n)+(1-\omega)h_a(y_n),\end{split}\]
where $\omega\in(0,\frac{1}{2})$ is a coupling parameter and $h_a\colon\mathbb R\to\mathbb R$ is the skew tent map
\[h_a(z)=\begin{cases}
az&\text{ if }z\leq a^{-1}\\
-\displaystyle{\frac{a}{a-1}(1-z)}&\text{ if }z\geq a^{-1},\ a>1.\end{cases}\]
%For certain parameters,
% If $a$ is fixed and $\omega$ is varied
% \textcolor{red}{(decrease?increase?)}, 
If $a>1$ and $a\neq2$, then as $\omega$ decreases,
the family undergoes {\it the blowout bifurcation}  at $\omega=\omega_b(a)$ %\cite[Lemma~2.2]{glendinning_2001} 
\cite{Gle99,HM97}.
Glendinning \cite[Theorem~A]{glendinning_2001} proved that
for any $a\in(\frac{1}{2}(1+\sqrt{5}),2)$ and any $\omega\in(\omega_b(a),\frac{1}{2a})$,
the set of $2$-unstable periodic points of the corresponding map is dense in its `topological attractor', a closed invariant set that strictly contains its `metric attractor', the diagonal $\{(x,x)\in[0,1]^2\colon x\in[0,1]\}$.
For these parameters, the coexistence of $1$-unstable periodic points and $2$-unstable periodic points in the metric attractor is known \cite{Gle99,glendinning_2001,HM97}.
For these parameters, the density of $1$-unstable periodic points in the topological attractor remains unknown. 
\begin{conjecture}For the coupled tent maps for which $a$ and $\omega$ satisfy the conditions in \cite[Theorem~A]{glendinning_2001}, the set of $1$-unstable periodic points is dense in its topological attractor.  \end{conjecture}
The definition of topological attractor and that of metric attractor can be found in \cite{Mil85}.
%\textcolor{cyan}{Returning to the heterochaos baker maps, a natural direction is to generalize Theorem~\ref{thm-1} to higher dimension. 
%\begin{conjecture}\label{con-II} For any integer $d\geq4$, there exists a piecewise affine map $F$ on $[0,1]^d$ with the following properties:  \begin{itemize} \item[(i)] $F$ has a dense orbit in $[0,1]^d$. \item[(ii)] For every $j\in\{1,\ldots,d-1\}$, the set of $j$-unstable periodic points is dense in $[0,1]^d$. \end{itemize} \end{conjecture} The meaning of $j$-unstable periodic points, $j\geq3$ is now clear. To fulfill (ii), we suspect that the central direction of $F$ must be multidimensional, unlike the heterochaos baker maps. We here mention one relevant result \cite{ABCDW} on periodic points of $C^1$ generic diffeomorphisms on manifolds of dimension $\geq3$. \textcolor{red}{$\geq4$? check} Clearly \cite{ABCDW} does not apply to maps with discontinuities. We were indebted to Toshi Sugiyama in formulating Conjecture~\ref{con-II}.}

\subsection{Measures of maximal entropy}\label{MME-sec}
%Topological entropy and measure-theoretic (Kolmogorov-Sina{\u\i}) entropy %of invariant measures  are connected by the variational principle.
Let $T$ be a %continuous 
Borel
map of a compact metric space $X$.
%and let $h_{\rm top}(T)$ denote its topological entropy \textcolor{red}{top entropy for non-continuous map}. 
 Let $M(X)$ denote the space of Borel probability measures on $X$ endowed with the weak* topology, $M(X,T)$  the set of elements of $M(X)$ that are $T$-invariant, and
 let $M^{\rm e}(X,T)$ denote the set of elements of $M(X,T)$ that are ergodic.
For  $\mu\in M(X,T)$ let $h(\mu)=h(T,\mu)$ denote the measure-theoretic  
entropy of $\mu$ with respect to $T$. %If $\varphi$ is continuous,
%The variational principle  \textcolor{red}{REF}
%asserts that
%\[h_{\rm top}(T)=\sup\{h(\mu)\colon\mu\in M(X,T)\}.\]
If $\sup\{h(\mu)\colon\mu\in M(X,T)\}$
is positive and finite, a measure that attains this supremum is called {\it a measure of maximal entropy} ({\it MME} for short). %The set of MMEs is a simplex whose extreme points are ergodic MMEs.
Let $M_{\rm max}^{\rm e}(X,T)$ denote the set of ergodic MMEs.
Central questions surrounding MMEs  
are their existence, finiteness and uniqueness. 
MMEs do not always exist %, see e.g., 
\cite{Buz14,Mis73}.
A system %with finite positive topological entropy 
that has the unique MME is called {\it intrinsically ergodic} \cite{W70}.
In the thermodynamic formalism \cite{Rue04},
the coexistence of MMEs is referred to as {\it phase transition}.

%For concrete examples of (piecewise) 
%smooth systems, we further require a  robustness under small perturbation of the system,
%as smooth systems may be models of %physical
%real world phenomena in time evolution, 
%Hands-on concrete examples of smooth systems for which MMEs robustly coexist on the same transitive set are in need, but the shortage %dearth of supply is quite apparent.

Each probability vector $(p_1,\ldots,p_k)$
$(k\geq2)$ uniquely determines
a shift invariant, ergodic
 Borel probability measure on the (one- or two-sided) full shift space on $k$-symbols. This measure is called a Bernoulli measure.
By the $(p_1,\ldots,p_k)$-Bernoulli system $(\sigma,\nu)$
we mean that $\sigma$ is the full shift on $k$-symbols,
and $\nu$ is the Bernoulli measure on the shift space associated with the probability vector $(p_1,\ldots,p_k)$.
%A {\it Bernoulli measure} is a  Borel probability measure on the one- or two-sided full shift space on $k$-symbols $(k\geq2)$ that is determined by a non-degenerate $k$-dimensional probability vector $(p_1,\ldots,p_k)$.
 %\textcolor{red}{Def incomplete \cite{Wal82}} 
% Bernoulli measures are shift invariant and ergodic, and define Bernoulli systems.
%We say a measurable dynamical system $(T,\nu)$, or simply $\nu$ is $(p_1,\ldots,p_k)$-{\it Bernoulli} if it is isomorphic to the $(p_1,\ldots,p_k)$-Bernoulli system. 
We say a $T$-invariant probability measure $\nu$ (or $(T,\nu)$) is $(p_1,\ldots,p_k)$-{\it Bernoulli} if $(T,\nu)$ is isomorphic to the $(p_1,\ldots,p_k)$-Bernoulli system.

Returning to the heterochaos baker map, let $p\colon[0,1]^3\to[0,1]^2$ denote the projection to the $(x_u,x_c)$-plane.
We denote the Lebesgue measure on $[0,1]^2$ and that on $[0,1]^3$ commonly by ${\rm Leb}$ despite the possible ambiguity. 
The next theorem shows
 a manifestation of the unstable dimension of variability in phase transition, i.e., {\it a heterochaos on ergodic level}. 
\begin{theorem}
\label{thm-a} 
For any $(a,b)\in \Delta$ the following statements hold:

\begin{itemize}
\item[(a)] There exist two $f_{a,b}$-invariant
 ergodic Borel probability measures $\mu_{\alpha}=\mu_{\alpha,a,b}$ and $\mu_{\beta}=\mu_{\beta,a,b}$ 
 of entropy $\log (M+1)$ that are $(\frac{1}{M+1},\ldots,\frac{1}{M+1})$-Bernoulli,  
charge any non-empty open subset of $[0,1]^3$ %fully supported on $[0,1]^3$ 
and satisfy
\[\mu_{\alpha,a,b}(\Omega_{\alpha_k})=
\mu_{\beta,a,b}(\Omega_{\beta_k})=\frac{1}{M+1}\ \text{ for  }k=1,\ldots,M;\]
\item[(b)] 
$M_{\rm max}^{\rm e}([0,1]^3,f_{a,b})=\{ \mu_{\alpha,a,b},\mu_{\beta,a,b}\}$;
\item[(c)] 
  $M^{\rm e}_{\rm max}([0,1]^2,f_a)=\{\mu_{\alpha,a,b}\circ p^{-1},\mu_{\beta,a,b }\circ p^{-1}\}$.
  \end{itemize}
  Moreover the following statement holds:
  \begin{itemize}
  \item[(d)]  $\mu_{\alpha,a,b}={\rm Leb}$ if and only if $(a,b)=(\frac{1}{M+1},\frac{1}{M(M+1)})$. Similarly,
 $\mu_{\beta,a,b}={\rm Leb}$ holds if and only if $(a,b)=(\frac{1}{M(M+1)},\frac{1}{M+1})$.
\end{itemize}\end{theorem}
%\textcolor{red}{The next theorem is the final main result of this paper, which asserts that the two ergodic MMEs in Theorem~\ref{thm-a} are singular with respect to the Lebesgue measure except for two special parameters.
%\begin{theorem}\label{thm-leb}The following statements hold: \begin{itemize} \item[(i)] $\mu_{\alpha,a,b}={\rm Leb}$ if and only if $(a,b)=(\frac{1}{M+1},\frac{1}{M(M+1)})$. \item[(ii)] $\mu_{\beta,a,b}={\rm Leb}$ if and only if $(a,b)=(\frac{1}{M(M+1)},\frac{1}{M+1})$.\end{itemize}
%\end{theorem}
%}
%Part of the relationship in Theorem~\ref{thm-leb} was used in \cite{T23} to deduce an exponential mixing for the Dyck shift from that for the heterochaos baker maps.

Item (a) for any $(a,b)\in \Delta$
was proved in \cite[Theorem~1.2]{TY23}. There it was also proved that 
$M_{\rm max}^{\rm e}([0,1]^3,f_{a,b})=\{ \mu_{\alpha},\mu_{\beta}\}$ provided  $a$ or $b$ belongs to the interval $(\frac{1}{M(M+1)},\frac{1}{M+1})$. In particular, this restriction 
excludes 
the `classical' parameter $(a,b)=(\frac{1}{3},\frac{1}{6})$ 
with $M=2$ investigated in \cite{STY21}. The novelty of Theorem~\ref{thm-a} is that the existence of precisely two ergodic MMEs of $f_{a,b}$ and $f_a$ as in (b) and (c) holds for {\it any} $(a,b)\in \Delta$ and {\it any} $a\in(0,\frac{1}{M})$. 
Item (d) implies that $\mu_{\alpha}$ or $\mu_{\beta}$ is a
  physical measure of $f_{a,b}$ if and only if 
  the parameter $(a,b)$ is these two particular values.
%  $(a,b)=(\frac{1}{M+1},\frac{1}{M(M+1)})$ or $(a,b)=(\frac{1}{M(M+1)},\frac{1}{M+1})$. 
For the definition of a physical measure, see the paragraph after Conjecture~\ref{con-Ber}.

To put Theorem~\ref{thm-a} into context, let us recall that
classical examples of 
intrinsically ergodic systems are transitive topological Markov shifts \cite{P64} and
 uniformly hyperbolic diffeomorphisms \cite{Bow75,Rue04,Sin72}. 
 A great deal of effort has been dedicated to verifying the intrinsic ergodicity for systems that are not uniformly hyperbolic, and
%fairly complete
satisfactory
pictures are emerging when the dimension of the ambient manifold is one and two. %see 
%\cite{Hof79} or \cite{BCS22} 
%respectively %for example,
%for the most comprehensive results. 
Among others, 
%One of the 
most comprehensive results are due to 
Hofbauer \cite{Hof79} who established the intrinsic ergodicity for a large class of piecewise monotonic interval maps, and 
  Buzzi et al. \cite{BCS22} who proved that any $C^\infty$ surface diffeomorphism has only finitely many ergodic MMEs, and the uniqueness holds when some transitivity condition holds.
%A global view in higher dimension is still incomplete %vague and  many important questions are wide open.
 
%In order to understand phase transitions,
It is completely relevant to investigate concrete examples of transitive systems that are not intrinsically ergodic:
to ask how many ergodic MMEs coexist and what are their geometric and statistical properties, and so on.
%To exclude trivial %examples, we require 
%the coexisting MMEs be supported on the same transitive set.
%, i.e., an invariant set containing a dense orbit. 
Most of known such examples 
are subshifts over finite alphabets
 (see e.g., \cite{GK18,Hay13,Kri74,KOR16,Pav16,Tho06}) or partially hyperbolic systems
(see e.g., \cite{BFT23,NMRV,RT22,RRTU12}). For any of the subshift examples,
we must delve into the structure of the shift space to investigate properties of the MMEs.
In \cite{BFT23,RRTU12}, coexisting MMEs do not appear explicitly but appear in abstract dichotomy theorems. In \cite{RT22}, two ergodic MMEs on $\mathbb T^4$ were constructed but the construction involves many geometric ingredients. In  \cite{NMRV} it was shown that a transitive endomorphism on $\mathbb S^1\times[0,1]$ constructed by Kan \cite{Kan94}
has three ergodic MMEs.

We would like to have easily accessible, simpler examples that 
display phase transitions. They should be piecewise affine, like the heterochaos baker maps. To exclude trivial examples we impose the transitivity.
%The above mentioned results \cite{BCS22,Hof79} seem to indicate that it is not easy to give a concrete example of piecewise smooth transitive systems that are not intrinsically ergodic.

%We mention such an interesting example with three ergodic MMEs. Kan \cite{Kan94} constructed a transitive diffeomorphism on the annulus $\mathbb S^1\times[0,1]$ that has two physical measures of intermingled basins, supported on the boundary of the annulus. Both measures are ergodic MMEs. N\'u\~nez-Madariaga et al. \cite{NMRV} found the third ergodic MME in Kan's example that is supported in the interior of the annulus and  obtained as the asymptotic distribution of periodic points in the interior of the annulus.
%These structures are robust, and  the existence of the boundary of the ambient manifold is essential for the robustness.

\begin{question}\label{q-MME}Does there exist a transitive, piecewise affine map that has three %more than two 
ergodic MMEs?\end{question}

%\begin{question}\label{q-MME}Does there exist a piecewise affine map that has more than two ergodic MMEs with overlapping supports?\end{question}
%We say two Borel probability measures $\mu_1$, $\mu_2$ have overlapping supports if $\mu_1({\rm supp}(\mu_1)\cap{\rm supp}(\mu_2))>0$

\subsection{Distributions of periodic points}
For a large class of dynamical systems that are intrinsically ergodic, periodic points equidistribute with respect to the unique MME. Not much is known on distributions of periodic points for systems that are not intrinsically ergodic.

Our numerical experiments  suggest that the distribution of $1$-unstable periodic points differs from that of $2$-unstable periodic points qualitatively, as shown in \textsc{Figure}~\ref{fig:periodicorbits}.
We conjecture that they equidistribute with respect to the two ergodic MMEs in Theorem~\ref{thm-a}.
 % The topology of convergence is in terms of the weak* topology on the set of Borel probability measures on $X$: the convergence of integrals of any continuous function. 

\begin{conjecture}\label{dist-con}%[Distributions of periodic points of different indices] 
For any $(a,b)\in \Delta$ and any continuous function $\phi\colon[0,1]^3\to\mathbb R$ we have  \[\lim_{n\to\infty}\frac{1}{\#{\rm Per}_{\alpha,n}(f)}\sum_{x\in{\rm Per}_{\alpha,n}(f)}\phi(x)=\int\phi {\rm d}\mu_{\alpha}\]  and 
\[ \lim_{n\to\infty}\frac{1}{\#{\rm Per}_{\beta,n}(f)}\sum_{x\in{\rm Per}_{\beta,n}(f)}\phi(x)=\int\phi {\rm d}\mu_{\beta}.\]  \end{conjecture}
%A positive solution to this conjecture will certainly advance our knowledge on distributions of periodic points in general three dimensional systems in which periodic points of different indices coexist.
For the classical parameter $(a,b)=(\frac{1}{3},\frac{1}{6})$ with $M=2$, the MME $\mu_\alpha$ coincides with the Lebesgue measure on $[0,1]^3$ (see Theorem~\ref{thm-a}(d)). If Conjecture~\ref{dist-con} is true, then $1$-unstable periodic points must equidistribute with respect to the Lebesgue measure
as their periods tend to infinity. The first row of
\textsc{Figure}~\ref{fig:periodicorbits} does not look supporting  Conjecture~\ref{dist-con}. Needless to say, numerical computations of periodic points of higher periods are more difficult.

%\textcolor{cyan}{Question~\ref{q-MME} might be related to Conjecture~\ref{con-II}. For example, if there were %exists  a `nice' piecewise affine map $F$ on $[0,1]^4$ for which the sets of $1$, $2$, $3$-unstable periodic points are all dense in $[0,1]^4$, then
%we expect that  these periodic points would equidistribute with respect to three different invariant probability measures respectively,  in the sense of Conjecture~\ref{dist-con}, and each of these measures %are  would be ergodic MMEs for $F$.}

\subsection{Parameter dependence of the ergodic MMEs}\label{para-sec}
As the parameter $(a,b)$ varies, the two ergodic MMEs in Theorem~\ref{thm-a} vary.
Only for the two special parameter values, one of them coincides with {\rm Leb} (see Theorem~\ref{thm-a}(d)).
A natural question to ask is how the two ergodic MMEs depend on the parameter. As a second main result of this paper we have obtained the continuous dependence. 
\begin{theorem}%[The continuity of the ergodic MMEs]
\label{thm-c}
\ %The following statements hold:
\begin{itemize}
\item[(a)] The maps
$(a,b)\in \Delta\mapsto\mu_{\alpha,a,b}\in M([0,1]^3)$ and $(a,b)\in \Delta\mapsto\mu_{\beta,a,b}\in M([0,1]^3)$ are continuous.

\item[(b)] For any $b\in(0,\frac{1}{M})$, 
the maps
$a\in (0,\frac{1}{M})\mapsto\mu_{\alpha,a,b}\circ p^{-1}\in M([0,1]^2)$ and $a\in (0,\frac{1}{M})\mapsto\mu_{\beta,a,b}\circ p^{-1}\in M([0,1]^2)$  are continuous.
\end{itemize}
\end{theorem}
The continuity in the statement of the theorem is with respect to the weak* topology on the space of Borel probability measures: for example, the continuity of the map $(a,b)\in\Delta\mapsto\int\varphi{\rm d}\mu_{\alpha,a,b}$
for any continuous function $\varphi\colon[0,1]^3\to\mathbb R$.

In view of Theorem~\ref{thm-c}, it is also natural to ask how smooth the dependence of the ergodic MMEs on the parameter is.
This question is within the scope of {\it linear response} in dynamical systems, which was explored in a few pioneering works \cite{KKPH89,Rue97,Rue98} for uniformly hyperbolic systems equipped with SRB measures.  
  %\cite{You02}. 
For the heterochaos baker maps and their ergodic MMEs, we conjecture the following.

\begin{conjecture}
For a nice class of continuous functions
$\varphi\colon[0,1]^3\to\mathbb R$, the maps
$(a,b)\in\Delta\mapsto\int\varphi{\rm d}\mu_{\alpha,a,b}$ and $(a,b)\in\Delta\mapsto\int\varphi{\rm d}\mu_{\beta,a,b}$
are differentiable. \end{conjecture}

%\begin{question} Is it possible to describe the natural coding space of the modified heterochaos baker map in \cite[Section~1.7]{STY21}? How many ergodic MMEs does it have?\end{question}

%The Hausdorff dimension of a Borel probability measure $\mu$ is \[\dim_{\rm H}(\mu)=\inf\{\dim_{\rm H}(A)\colon A\subset [0,1]^3, \mu(A)=1\}.\]
%\begin{conjecture}For any $(a,b)\in \Delta$ we have $\dim_{\rm H}(\mu_{\alpha,a,b})<3$ and $\dim_{\rm H}(\mu_{\beta,a,b})=3$, $\dim_{\rm H}(\mu_{\alpha,a})<2$ and $\dim_{\rm H}(\mu_{\beta,a})=2$.\end{conjecture}
 
%Lai Sang Young, ETDS

%\begin{theorem} For any continuous function $\varphi\colon[0,1]^3\to\mathbb R$ such that $\varphi=\varphi\circ\iota$ and $\sup_{[0,1]^3}\varphi-\inf_{[0,1]^3}\varphi<\log(M+1)-\sqrt{4M}$, there exist at least two ergodic equilibrium states. \end{theorem}

%\begin{problem} Describe the largest possible class of potential functions for which the equilibrium state is not unique. \end{problem} \begin{proof}Boundary\end{proof}

\subsection{Ergodicity, mixingness of the Lebesgue measure}\label{erg-sec}
The heterochaos baker maps are variants of the well-known baker map given by
\[(x,y)\in[0,1)^2\mapsto\begin{cases}\vspace{1mm}
    \displaystyle{\left(2x,\frac{y}{2}\right)}&\text{ on }\displaystyle{\left[0,\frac{1}{2}\right)\times[0,1)},\\
   \displaystyle{\left (2x-1,\frac{y+1}{2}\right)}&\text{ on }\displaystyle{\left[\frac{1}{2},1\right)\times[0,1)}.
   \end{cases}\]
See \textsc{Figure}~\ref{bakermap}.
The name `baker' is used since the action of the map is reminiscent of the kneading dough \cite{H56}.
The baker map preserves the Lebesgue measure on $[0,1)^2$, which is Bernoulli, and 
has exponential decay of correlations  
for H\"older continuous functions with respect to the Lebesgue measure \cite{Bow75,Rue04}.
We would like to recover these nice statistical properties for the heterochaos baker maps.

Before proceeding further let us recall a few basic notions in ergodic theory. Let $T\colon X\to X$ be a measurable dynamical system preserving a probability measure $\nu$. Recall that $(T,\nu)$ is {\it ergodic} if for any pair $A$, $B$ of measurable sets in $X$ we have
\[\lim_{n\to\infty}\frac{1}{n}\sum_{i=0}^{n-1}\nu(A\cap T^{-i}(B))=\nu(A)\nu(B).\] We say $(T,\nu)$ is {\it weak mixing} if for any pair $A$, $B$ of measurable sets in $X$ we have
\[\lim_{n\to\infty}\frac{1}{n}\sum_{i=0}^{n-1}|\nu(A\cap T^{-i}(B))-\nu(A)\nu(B)|=0.\]
It is well-known \cite{Wal82} that $(T,\nu)$ is weak mixing if and only if the product system $(T\times T,\nu\times\nu)$ is ergodic.
Let $k\geq2$ be an integer.
We say $(T,\nu)$ is {\it $k$-mixing} if 
for all measurable sets $A_0,A_1\ldots,A_{k-1}$ in $X$ we have 
\[\lim_{n_1,\ldots,n_{k-1}\to\infty}\nu(A_0\cap T^{-n_1}(A_1)\cap\cdots\cap T^{-n_1-n_2-\cdots-n_{k-1}}(A_{k-1}))=\prod_{j=0}^{k-1}\nu(A_j).\]
$2$-mixing is usually called {\it mixing} or {\it strong mixing}. 
Clearly $(k+1)$-mixing implies $k$-mixing, but the converse is unknown.
%The notion of $k$-mixing was introduced by Rohlin \cite{R49}, and he conjectured that $2$-mixing implies $3$-mixing.
%Ledrappier \cite{Led78} gave a counterexample for a more general question in the setting of $\mathbb Z^d$-actions. 
We say $(T,\nu)$ is {\it mixing of all orders} if it is $k$-mixing for any $k\geq2$. We summarize the well-known implications  \cite{R63,Wal82}: Bernoulli $\Longrightarrow$ $K$ (or exact) $\Longrightarrow$ mixing of all orders $\Longrightarrow$ mixing $\Longrightarrow$ weak mixing $\Longrightarrow$ ergodic.
The definition of $K$-property and that of exactness can be found in \cite{Wal82}.

\begin{figure}%[b]
\begin{center}
\includegraphics[height=4cm,width=9cm]{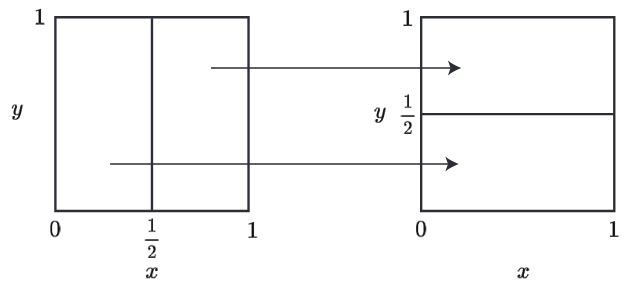}
%{chaos2d.eps}
\caption{The baker map: The $x$-direction is expanding by factor $2$ and the $y$-direction is contracting by factor $\frac{1}{2}$. }\label{bakermap}
\end{center}
\end{figure}

The Lebesgue measure is especially important in the viewpoint that observable events 
correspond to sets of positive Lebesgue measure. One can check that ${\rm Leb}$ is $f_a$-invariant for any $a\in(0,\frac{1}{M})$, and
${\rm Leb}$ is $f_{a,b}$-invariant if and only if $a+b=\frac{1}{M}$ (see Lemma~\ref{preserve}).
 We put
\begin{equation}\label{ga}g_a=f_{a,\frac{1}{M}-a},\end{equation}
and investigate statistical properties of 
$f_a$ and $g_a$ with respect to ${\rm Leb}$.  
 Since the $x_s$-direction is uniformly contracting under the forward iteration of $g_a$,
statistical properties of $(g_a,{\rm Leb})$ are almost identical to that of $(f_a,{\rm Leb})$.
For this reason, we will
write $F=F_a$ for both $f_a$ and $g_a$, and write $[0,1]^d$ for both $[0,1]^2$ and $[0,1]^3$ respectively.  This means that statements about $F_a\colon[0,1]^d\to[0,1]^d$ pertain to both $f_a\colon[0,1]^2\to[0,1]^2$ and $g_a\colon[0,1]^3\to[0,1]^3$.

The ergodicity of the fully branched interval map $\tau_a\colon[0,1]\to[0,1]$ with respect to the Lebesgue measure on $[0,1]$ implies that
  \begin{equation}\label{chia}\lim_{n\to\infty}\frac{1}{n}\sum_{i=0}^{n-1}\phi^c(g_a^i(x))=(1-2Ma)\log M\ \text{ for Lebesgue a.e. $x\in[0,1]^3$.}\end{equation}
We classify $F_a$ into three types
according to the sign of this limit value:
\begin{itemize}
\item $a\in(0,\frac{1}{2M})$ (mostly expanding center);

 \item $a\in(\frac{1}{2M},\frac{1}{M})$
 (mostly contracting center);
 \item $a=\frac{1}{2M}$ (mostly neutral center).
 \end{itemize}
The limit value in \eqref{chia} will be denoted by $\chi^c(g_a,{\rm Leb})$ and called {\it the central Lyapunov exponent} (see Section~\ref{lyapunov-sec} for the definition of Lyapunov exponents). 
The analysis of statistical properties of the maps with mostly neutral center is more difficult
than that of other maps.
%For this reason, let us make the following convention: When we give a statement on $g_a$, it also means the corresponding statement on $f_a$ implicitly. For example, the stament `$(g_a,{\rm Leb})$ is ergodic.' means that both  $(g_a,{\rm Leb})$ and $(f_a,{\rm Leb})$ are ergodic.

The next theorem summarizes results on mixing properties of $(F_a,{\rm Leb})$.
Item (b) is stronger than (a).
 Item (c) strengthens (b) with the exception of the mostly neutral parameter.
\begin{theorem}\label{erg-thm}   \

\begin{itemize}
\item[(a)] {\rm (\cite{STY21}, Theorem~1.2)}. For any $a\in (0,\frac{1}{M})\setminus\{\frac{1}{2M}\}$, $(F_a,{\rm Leb})$ is weak mixing. 
\item[(b)]
{\rm (\cite{T23}, Theorem~A).}
For any $a\in(0,\frac{1}{M})$, $(F_a,{\rm Leb})$ is mixing of all orders.

\item[(c)]\label{Bernoulli-thm} For any $a\in (0,\frac{1}{2M})$, $(F_a,{\rm Leb})$ 
is
$(\frac{1-Ma}{M},\ldots,\frac{1-Ma}{M},Ma)$-Bernoulli. For any $a\in (\frac{1}{2M},\frac{1}{M})$, $(F_a,{\rm Leb})$ is %$(\underbrace{a,\ldots,a}_{M},1-Ma)$
$(a,\ldots,a,1-Ma)$-Bernoulli. 
\end{itemize}
\end{theorem}

% \begin{theorem}[\cite{STY21}]%[\cite{STY21}, Theorem~1.2] \label{erg-thm} For any $a\in (0,\frac{1}{M})\setminus\{\frac{1}{2M}\}$, $(F_a,{\rm Leb})$ is weak mixing. \end{theorem}

%\begin{theorem}[\cite{T23}, Theorem~A]\label{mix-thm} For any $a\in(0,\frac{1}{M})$, $(F_a,{\rm Leb})$ is mixing of all orders.\end{theorem}

%As a corollary to Theorem~\ref{erg-thm}(ii)  we obtain the following statement.
Let ${\rm dist}(\cdot,\cdot)$ denote the Euclidean distance on $[0,1]^2$ or $[0,1]^3$.
We say a pair $(x,y)$ of distinct points in $[0,1]^d$ is {\it a Li-Yorke pair} of $F$ if
\[\liminf_{n\to\infty} {\rm dist}(F^n(x),F^n(y))=0
\ \text{ and }\ \limsup_{n\to\infty} {\rm dist}(F^n(x),F^n(y))=\sqrt{d}.\]
\begin{cor}For any $a\in(0,\frac{1}{M})$, Lebesgue almost every pair of points in $[0,1]^d$ is a Li-Yorke pair of $F_a$.\end{cor}

 \begin{proof}
 Since $(F,{\rm Leb})$ is weak mixing by Theorem~\ref{erg-thm}(b), the desired conclusion follows from Birkhoff's ergodic theorem applied to $(F\times F,{\rm Leb}\times{\rm Leb})$.
   \end{proof}

The proof of Theorem~\ref{erg-thm}(a) in \cite{STY21} amounts to showing the ergodicity of
    the associated product system using a variant of
   {\it Hopf's argument}
     that is used for proving the ergodicity of volume-preserving Anosov diffeomorphisms. 
 Although only the special case $(a,b)=(\frac{1}{3},\frac{1}{6})$ with $M=2$ was treated in \cite{STY21}, essentially the same proof works to cover general cases
 except for the mostly neutral parameter $\frac{1}{2M}$.
In fact, the exactness of
$(f_a,{\rm Leb})$ and the $K$-property of an invertible restriction of $(g_a,{\rm Leb})$ to a set of measure $1$ were shown
in \cite{T23}.

Let us comment on 
one interesting consequence of Theorem~\ref{erg-thm}(b). 
It implies that
$(F_a,{\rm Leb})$ is ergodic for all $a\in(0,\frac{1}{M})$.
In particular, {\it ${\rm Leb}$
continues to be the unique physical measure of $F_a$ for all $a\in(0,\frac{1}{M})$,
 despite the change of sign of its central Lyapunov exponent at $a=\frac{1}{2M}$.}
A Borel probability measure $\mu$ invariant under a Borel map $T$ of 
 a compact Riemannian manifold $X$ is said to be {\it physical} if 
 there exists a subset $E$ of $X$ of positive Lebesgue measure such that for any $x\in E$ we have
\[\lim_{n\to\infty}\frac{1}{n}\sum_{i=0}^{n-1}\phi(T^i(x))=\int\phi{\rm d}\mu\ \text{ for any continuous $\phi\colon X\to\mathbb R$}.\]

In one-dimensional dynamics, 
the change of sign of Lyapunov exponents of ergodic physical measures often leads to non-smooth changes of physical measures.
For the quadratic map $x\in[0,1]\mapsto ax(1-x)\in[0,1]$ $(0\leq a\leq 4)$, the physical measure with positive Lyapunov exponent is often supported on the union of intervals, while 
%implies a chaotic behavior, and 
the physical measure with negative Lyapunov exponent is supported on a single periodic orbit
%negative Lyapunov exponent implies a periodic behavior
%the existence of a hyperbolic attracting periodic orbit
(see e.g., \cite{dMevSt93}). 
Tsujii and Zhang \cite{TZ} asked if 
ergodic physical measures of partially hyperbolic systems 
can `vary smoothly' with parameter, changing the sign of their %Lyapunov exponents in central directions. 
central Lyapunov exponents.
In \cite{TZ} they 
constructed an open set of parametrized families of local diffeomorphisms of $\mathbb T^2$ that {\it robustly} display this phenomenon, thereby answering their question affirmatively.  
Albeit non-robust,
the family $(F_a)_a$ displays this sort of phenomenon.

 It is reasonable to expect that the Bernoulliness of the Lebesgue measure in Theorem~\ref{erg-thm}(c) extends to the mostly neutral parameter in a continuous fashion.
 \begin{conjecture}\label{con-Ber}
$(F_{\frac{1}{2M}},{\rm Leb})$ is $(\frac{1}{2M},\ldots,\frac{1}{2M},\frac{1}{2})$-Bernoulli.
 \end{conjecture}

%\begin{theorem}
%\label{Bernoulli-thm}
%The following statements hold:
%\begin{itemize}
%\item[(i)] For any $a\in (0,\frac{1}{2M})$, $(F_a,{\rm Leb})$ is
%$(\frac{1-Ma}{M},\ldots,\frac{1-Ma}{M},Ma)$-Bernoulli.
%\item[(ii)] For any $a\in (\frac{1}{2M},\frac{1}{M})$, $(F_a,{\rm Leb})$ is %$(\underbrace{a,\ldots,a}_{M},1-Ma)$
%$(a,\ldots,a,1-Ma)$-Bernoulli. 
%\end{itemize}
%\end{theorem}

\subsection{Rates of mixing}\label{rates-sec}

%All functions appearing in this paper are real-valued.
Let $T\colon X\to X$ be a measurable dynamical system preserving a probability measure $\nu$.
For a pair $(\varphi,\psi)$ of real-valued functions in $L^2(\nu)$,
define their correlations by
\[{\rm Cor}_n(T;\varphi,\psi;\nu)=\left|\int\varphi(\psi\circ T^n){\rm d}\nu-\int\varphi{\rm d}\nu\int\psi{\rm d}\nu\right|\text{ for }n\geq1.\]
Then $(T,\nu)$ is mixing if and only if
for any pair $(\varphi,\psi)$ of functions in $L^2(\nu)$,
${\rm Cor}_n(T;\varphi,\psi;\nu)\to0$ as $n\to\infty.$
As well-known examples 
show, the rate of this decay can become arbitrarily slow
even if $(T,\nu)$ is mixing.
So, one usually sets %an appropriate 
%problem is to set 
a suitable function space and considers rates of decay of correlations for functions that belong to this function space.

We review the result in \cite{T23} on rates of mixing for the heterochaos baker maps. 
%Let $X$ be a metric space. 
%for $\eta\in(0,1]$ let $\mathscr H_\eta(X)$ denote the set of H\"older continuous functions on $X$ with exponent $\eta$.
We say $(F,{\rm Leb})$ is {\it exponentially mixing}
if for each $\eta\in(0,1]$
there exists $\lambda=\lambda(\eta)\in(0,1)$ such that 
for any pair $(\varphi,\psi)$ of real-valued 
H\"older continuous functions on $[0,1]^d$ with exponent $\eta$, there exists $C=C(\varphi,\psi)>0$ such that 
${\rm Cor}_n(F;\varphi,\psi;{\rm Leb})\leq C\lambda^n$ for all $n\geq1$.
\begin{theorem}[\cite{T23}, Theorem~B]\label{thm-exp-mix}
For any $a\in(0,\frac{1}{M})\setminus\{\frac{1}{2M}\}$,
$(F_a,{\rm Leb})$ is exponentially mixing.
\end{theorem}A proof of the exponential mixingness in Theorem~\ref{thm-exp-mix} for each $a\in(0,\frac{1}{2M})$ amounts to constructing a `tower', i.e., a uniformly expanding induced Markov map with infinitely many branches equipped with an inducing time $R$ with exponential tails, and then apply 
the general result of Young \cite{You98}. 
Due to the symmetry of $g_a=f_{a,\frac{1}{M}-a}$ with respect to the parameter and the symmetry of correlations, the exponential mixing for $a\in(\frac{1}{2M},\frac{1}{M})$ follows as a consequence. Meanwhile, for the mostly neutral parameter $a=\frac{1}{2M}$ we have
$\sum_{n}{\rm Leb}\{R>n\}=\sum_nn^{-1/2}=\infty$, and so Young's general result \cite{You99} on subexponential decay of correlations
does not apply.
See Conjecture~\ref{polynomial-con} and related numerical results below.

% Like the heterochaos baker maps, skew product maps over uniformly hyperbolic or expanding bases are considered to be simple models of partially hyperbolic systems.

 The heterochaos baker maps may even be viewed as the simplest partially hyperbolic systems: for substantial parameter ranges, the uniform expansion in the $x_u$-direction dominates the maximal expansion in the $x_c$-direction. On the concept of partial hyperbolicity, we refer the reader to \cite{BDV04,BP74,P04} for example.
%Since the heterochaos baker maps are one of the simplest models of partially hyperbolic systems, 
% It is relevant to compare our results with others on the exponential decay of correlations for partially hyperbolic systems. 
Young's method \cite{You98,You99}
of deducing nice statistical properties 
using towers with fast decaying tails has  been successfully implemented for 
some partially hyperbolic systems, e.g.,
\cite{ALP05,C02,C04,Do00} to deduce exponential decay of correlations.
% As applications of \cite{You98,You99}
%to partially hyperbolic systems which yield
%an exponential decay of correlations for physically relevant measures, we mention
% \cite{ALP05,C02,C04,Do00} for example.
With more functional analytic methods,
exponential decay of correlations was proved for 
certain partially hyperbolic diffeomorphisms \cite{CN17}, and
 for certain piecewise partially hyperbolic endomorphisms \cite{BBL20}.  %and various other limit theorems were proved in \cite{Do04}.
%See \cite{CL22} and the referenfce therein for a state of the art of a general theory on a functional analytic approach \textcolor{red}{2-dim}.
Clearly the heterochaos baker maps are covered by none of these existing results.

\subsection{Numerical computations of autocorrelations}
 Since the $x_c$-direction is most important, the coordinate function 
$\phi\colon (x_u,x_c,x_s)\in[0,1]^3\mapsto x_c\in[0,1]$ should capture an essential statistical information on the dynamics of $f_{a,b}$. 
For $M=2$
we numerically compute autocorrelations of $\phi$ from its typical time series.
Pick $x\in[0,1]^3$ and define $X=(X_i)_{i\geq0}$ by 
$X_i=\phi(f^i(x))$.
For a delay $n\ge 1$ and an orbit length $N\geq1$ define \[{\rm Cor}_{n,N}(X)=\left|\frac{1}{N}\sum^{N-1}_{i=0}(X_i-E_N(X))(X_{i+n}-E_N(X))\right|,\]
 where $E_N(X)=\frac{1}{N}\sum_{i=0}^{N-1}X_i$.
 By Birkhoff's ergodic theorem,  
  ${\rm Cor}_{n,N}(X)$
  converges to 
 ${\rm Cor}_n(f;\phi,\phi;{\rm Leb})$ as $N\to\infty$ for Lebesgue almost every $x\in[0,1]^3$.

 For certain ranges of $n$ we have computed
${\rm Cor}_{n,N}(X)$ with $N\gg n$.
The left panel in \textsc{Figure~\ref{fig:decay}} shows the exponential decay
as in Theorem~\ref{thm-exp-mix}. The right panel seems to indicate the polynomial decay of order $n^{-3/2}$. We put this as a conjecture.
\begin{conjecture}\label{polynomial-con}
For any pair $(\varphi,\psi)$ of H\"older continuous functions on $[0,1]^d$,
there exists a constant $C(\varphi,\psi)>0$ such that
\[{\rm Cor}_n(F_{\frac{1}{2M}};\varphi,\psi;{\rm Leb})\leq C(\varphi,\psi) n^{-3/2}\ \text{ for every }n\geq1.\] 
Moreover this estimate is optimal, i.e.,
there exist a pair $(\varphi,\psi)$ of H\"older continuous functions on $[0,1]^d$ and a constant 
$C'(\varphi,\psi)>0$ such that
\[{\rm Cor}_n(F_{\frac{1}{2M} };\varphi,\psi;{\rm Leb})\geq C'(\varphi,\psi) n^{-3/2}\ \text{ for every }n\geq1.\]
\end{conjecture}

\begin{figure}%[b]
    \begin{center}
    \includegraphics[]{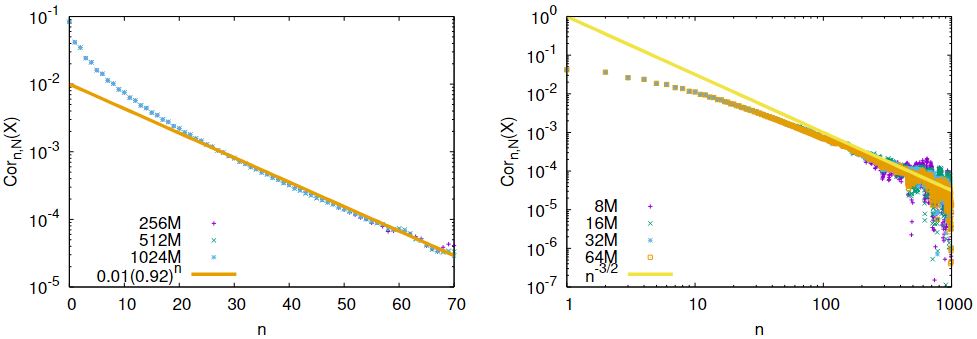}
    \end{center}
    \caption{ 
    Decay of autocorrelations for
    $\phi\colon (x_u,x_c,x_s)\in[0,1]^3\mapsto x_c\in[0,1]$, $a=\frac{1}{3}$ (left) and $a=\frac{1}{4}$ (right) with $M=2$.
The left panel shows an exponential decay 
${\rm Cor}_{n,N}(X)\approx0.01(0.92)^n$, 
whereas the right one shows a
    polynomial decay 
${\rm Cor}_{n,N}(X)\approx n^{-3/2}$.
The orbit length $N$ is 
$256\times 10^6$, $512\times 10^6$, $1024\times 10^6$ for the left  and $8\times 10^6$,  $16\times 10^6$, $32\times 10^6$,  $64\times 10^6$ for the right. 
The initial conditions of the orbits have been chosen according to the random number generator.
All the axes but the horizontal one in the left panel are in the logarithmic scale.}
\label{fig:decay}
\end{figure}

\section{Preliminaries for proofs of the main results}
Contents of
this section are preliminaries needed for the proofs of the main results presented in Section~2. Starting with the definition
of the Dyck shift
in Section~\ref{Dyck},
 we delve into its structure in Section~\ref{str-1-sec} and Section~\ref{embed-sec}.
 In Section~\ref{code-sec} we state the connection between the heterochaos baker maps and the Dyck shift that underlies the proofs of the main results. In Section~\ref{str-code} we analyze the structure of coding maps. 
  In Section~\ref{str-mme} we recall the construction of the two ergodic MMEs for the Dyck shift and that for the heterochaos baker maps. In Section~\ref{lyapunov-sec} we introduce Lyapunov exponents for invariant measures of $f_{a,b}$ and investigate their properties. In Section~\ref{projection-sec} we clarify a correspondence between invariant measures of $f_a$ and that of $f_{a,b}$. In Section~\ref{leb-section} we clarify for which parameters the Lebesgue measure is invariant.

\subsection{The Dyck shift}\label{Dyck}
Let $\mathbb N_0=\mathbb N\cup\{0\}$.  
Let $S$ be a non-empty finite discrete set and let $S^{\mathbb N_0}$, $S^{\mathbb Z}$ denote one- or two-sided Cartesian product topological space of $S$ respectively. The left shift acts continuously on these two spaces.
  A {\it subshift} on $S$ is a shift invariant closed subset of $S^{\mathbb N_0}$ or    $S^{\mathbb Z}$.     
    For a subshift $\Sigma$,
 let $L(\Sigma)$
 denote the set of finite words in $S$ that appear in some elements of $\Sigma$.
 For convenience, we include the empty word $\emptyset$ in  $L(\Sigma)$
 and let $\emptyset \xi=\xi\emptyset=\xi$ for  $\xi\in L(\Sigma)\setminus\{\emptyset\}$.
Elements of $L(\Sigma)\setminus\{\emptyset\}$ are called {\it admissible words}.
For a two-sided subshift $\Sigma$
and $j\in\mathbb Z$, $n\in\mathbb N$, $\xi=\xi_1\cdots\xi_n\in L(\Sigma)$, we 
define 
\begin{equation}\label{cylinder-def}\Sigma(j;\xi)=\{(\omega_i)_{i\in\mathbb Z}\in\Sigma\colon \omega_{i+j-1}=\xi_{i}\ \text{ for  }i=1,\ldots, n\}.\end{equation}

The Dyck shift has its origin in the theory of languages \cite{AU68}.
Krieger \cite{Kri00} introduced a certain class of subshifts having some algebraic property, called property A subshifts. The Dyck shift is a fundamental shift space in this class. 
 It is a subshift on the alphabet 
$D$ in \eqref{D}
consisting of $M$ brackets, $\alpha_i$ left and $\beta_i$ right in pair,
whose admissible words are words of legally aligned brackets.
%An infinite sequence of these brackets belong to the shift space if and only if all brackets in the sequence are legally aligned.
% For a point $x$ in the shift space $\Sigma^\mathbb Z$,
%$x$ is in $D$ if and only if every finite block (word) appearing in $x$ has a nonzero reduced form. 
To be more precise,
let $D^*$ denote the set of finite words in $D$.
%including the empty word $\emptyset$.
Consider the monoid with zero, with $2M$ generators in $D$ and the unit element
$1$ with relations 
\[\alpha_i\cdot\beta_j=\delta_{ij},\
0\cdot 0=0\text{ and }\] \[\gamma\cdot 1= 1\cdot\gamma=\gamma,\
 \gamma\cdot 0=0\cdot\gamma=0
\text{ for }\gamma\in D^*\cup\{ 1\},\]
where $\delta_{ij}$ denotes Kronecker's delta.
For $k\in\mathbb N$ and $\gamma_1\gamma_2\cdots\gamma_k\in D^*$ %\setminus\{\emptyset\}$, 
let
\[{\rm red}(\gamma_1\cdots\gamma_k)=\prod_{j=1}^k\gamma_j.\]
%and put \[{\rm red}(\emptyset)=1\textcolor{red}{necessary for us?}.\]
The one- and two-sided Dyck shifts on $2M$ symbols are defined by
\[\begin{split}
\Sigma_{D}^+&=\{\omega\in D^{\mathbb N_0}\colon {\rm red}(\omega_j\cdots \omega_k)\neq0\text{ for }j,k\in\mathbb N_0\text{ with }j<k\},\\
\Sigma_{D}&=\{\omega\in D^{\mathbb Z}\colon {\rm red}(\omega_j\cdots \omega_k)\neq0\text{ for }j,k\in\mathbb Z\text{ with }j<k\},\end{split}\]
respectively.

Another way to define $\Sigma_D^+$ and $\Sigma_D$ is the following. Consider a labeled directed graph that consists of infinitely many vertices $V_{ij}$, with $i\in\mathbb N_0$, $j\in\{1,\ldots,M^i\}$
together with edges each labeled with a symbol in $D$.
Each vertex $V_{ij}$ has 
$M$ outgoing edges
 \[V_{ij}\stackrel{\alpha_1}{\longrightarrow}V_{i+1,Mj-M+1},\ \ldots,\  V_{ij}\stackrel{\alpha_{M-1}}{\longrightarrow}V_{i+1,Mj-1},\ V_{ij}\stackrel{\alpha_M}{\longrightarrow}V_{i+1,Mj}\]
 and $M$ incoming edges
 \[V_{ij}\stackrel{\beta_1}{\longleftarrow}V_{i+1,Mj-M+1},\ \ldots,\ V_{ij}\stackrel{\beta_{M-1}}{\longleftarrow}V_{i+1,Mj-1},\ V_{ij}\stackrel{\beta_M}{\longleftarrow}V_{i+1,Mj}.\]
 The bottom vertex $V_{01}$ has additional $M$ loop edges 
 \[V_{01}\stackrel{\beta_1}{\longrightarrow}V_{01},\ \ldots,\  V_{01}\stackrel{\beta_{M-1}}{\longrightarrow}V_{01},\ V_{01}\stackrel{\beta_M}{\longrightarrow}V_{01}.\]
Then $\Sigma_D^+$ coincides with the set of one-sided infinite sequences of elements of $D$ that are associated with the infinite labeled paths in this graph starting at $V_{01}$, and $\Sigma_D$ coincides with the invertible extension of $\Sigma_D^+$: 
\[\Sigma_{D}=\{\omega=(\omega_i)_{i\in\mathbb Z}\in D^{\mathbb Z}\colon \omega_j\omega_{j+1}\cdots\in\Sigma_D^+ \text{ for all }j\in\mathbb Z\}.\]

\textsc{Figure}~\ref{d-dyck} shows 
part of a labeled directed graph associated with the Dyck shift with $M=2$. The set %$L(\Sigma_D^+)\setminus\{\emptyset\}=L(\Sigma_D)\setminus\{\emptyset\}$ 
of admissible words coincides with the set of strings of labels associated with the finite paths in the graph starting at the bottom vertex. Since $\Sigma_D^+$ and $\Sigma_D$ have infinitely many forbidden words, they are not topological Markov shifts.

Let $\sigma$ denote the left shift acting on $D^{\mathbb N_0}$ or $D^\mathbb Z$. Clearly we have 
$\sigma(\Sigma_D^+)=\Sigma_D^+$ and $\sigma(\Sigma_D)=\Sigma_D$. For simplicity, we still use the letter $\sigma$ to denote the restriction of the left shift to $\Sigma_D^+$ or $\Sigma_D$. Krieger \cite{Kri74} proved the existence of exactly two ergodic MMEs for the two-sided Dyck shift. The corresponding statement holds for the one-sided Dyck shift.

\begin{theorem}[\cite{Kri74}]\label{k-thm}
%For each integer $M\geq2$ 
There exist exactly two shift invariant ergodic Borel probability measures of maximal entropy $\log(M+1)$ for $(\Sigma_D,\sigma)$. 
They are $(\frac{1}{M+1},\ldots,\frac{1}{M+1})$-Bernoulli and 
 charge any non-empty open subset of  $\Sigma_D$.
\end{theorem}

\subsection{Structure of ergodic measures for the two-sided Dyck shift}\label{str-1-sec}
In this and the next subsections, following Krieger \cite{Kri74} we delve into the structure of the two-sided Dyck shift.
A key ingredient is a family of functions $H_i\colon \Sigma_D\to\mathbb Z$ ($i\in\mathbb Z$) given by
      \begin{equation}\label{H-def}H_i(\omega)=\begin{cases}\sum_{j=0}^{i-1} \sum_{k=1}^{M}(\delta_{\alpha_k,\omega_j}-\delta_{\beta_k,\omega_j})&\text{ for }  i\geq1,\\\sum_{j=i}^{-1} \sum_{k=1}^{M}(\delta_{\beta_k,\omega_j}-\delta_{\alpha_k,\omega_j})&\text{ for } i\leq -1,\\
      0&\text{ for }i=0.\end{cases}\end{equation}
For $i$, $j\in\mathbb Z$ define \[\{H_i=H_j\}=\{\omega\in\Sigma_D\colon H_{i}(\omega)=H_j(\omega)\}.\]

To describe ergodic shift invariant measures, we introduce three pairwise disjoint shift invariant Borel sets:
   \[\begin{split}A_0&=\bigcap_{i=-\infty }^{\infty}\left(\left(\bigcup_{
   j=1}^\infty\{ H_{i+j}=H_i\}\right)\cap\left(\bigcup_{j=1}^\infty\{ H_{i-j}=H_i\}\right)\right);\\
A_\alpha&=\left\{\omega\in\Sigma_D\colon
\lim_{i\to\infty}H_i(\omega)
=\infty\ \text{ and } \ \lim_{i\to-\infty}H_i(\omega)=-\infty\right\};\\
A_\beta&=\left\{\omega\in\Sigma_D\colon
\lim_{i\to\infty}H_i(\omega)=-\infty\ \text{ and } \
\lim_{i\to-\infty}H_i(\omega)=\infty\right\}.\end{split}\]

\begin{lemma}[\cite{Kri74}, pp.102--103]\label{trichotomy}
If $\nu\in M^{\rm e}(\Sigma_D,\sigma)$, then either $\nu(A_0)=1$, 
$\nu(A_\alpha)=1$ or $\nu(A_\beta)=1$.
\end{lemma}

\begin{figure}
\begin{center}
\includegraphics[height=7cm,width=11cm]{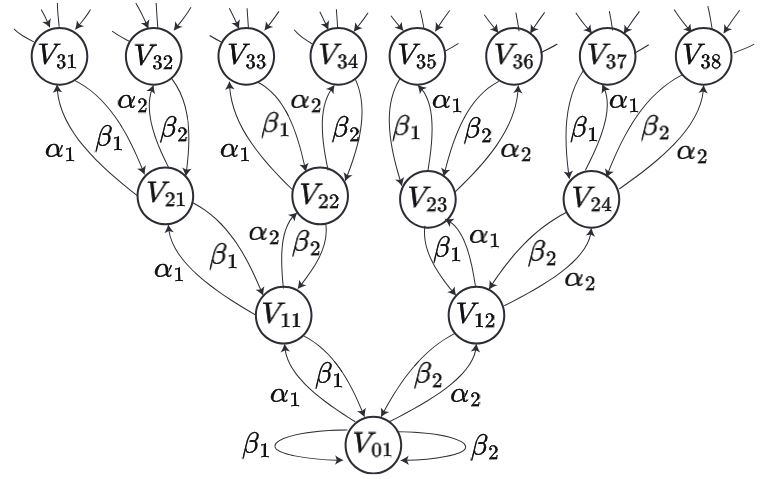}
%{direct-dyck-revised.eps}
\caption
{Part of a labeled directed graph associated with the Dyck shift with $M=2$. The upward edges are labeled with $\alpha_1$ or $\alpha_2$. The downward edges are labeled with $\beta_1$ or $\beta_2$.}\label{d-dyck}
\end{center}
\end{figure}

\subsection{Measure-theoretic embeddings of the full shift }\label{embed-sec} 
We introduce two different full shifts on $M+1$ symbols
 \[\Sigma_\alpha=\{\alpha_1,\ldots,\alpha_{M},\beta\}^{\mathbb Z}\quad\text{and}\quad\Sigma_\beta=\{\alpha,\beta_1,\ldots,\beta_M\}^{\mathbb Z}.\]
  Let $\sigma_\alpha$, $\sigma_\beta$ denote the left shifts acting on $\Sigma_\alpha$, $\Sigma_\beta$ respectively.
%For $i\in\mathbb Z$ and $\gamma\in D$ we put
%\[\Sigma_{D,i}(\gamma)=\{\omega\in\Sigma_D\colon\omega_i=\gamma\}.\]
With the notation in \eqref{cylinder-def}
we introduce two shift invariant Borel sets: 
\begin{equation}\label{2-sets}\begin{split}B_\alpha&=\bigcap_{i=-\infty}^\infty\bigcup_{k=1}^M\left(\Sigma_{D}(i;\alpha_k)\cup\left(\Sigma_{D}(i;\beta_k)\cap\bigcup_{j=1}^\infty\{ H_{i-j+1}=H_{i+1}\}\right)\right);\\
B_\beta&=\bigcap_{i=-\infty }^\infty\bigcup_{k=1}^M\left(\Sigma_{D}(i;\beta_k)\cup
\left(\Sigma_{D}(i;\alpha_k)\cap \bigcup_{j=1}^\infty\{H_{i+j}=H_i\}\right)\right).\end{split}\end{equation}
%Note that $B_\alpha\supset\{\alpha_1,\ldots,\alpha_M\}^\mathbb Z$ and $B_\beta\supset\{\beta_1,\ldots,\beta_M\}^\mathbb Z$. 
The set $B_\alpha$ (resp. $B_\beta$)
is precisely the set of sequences in $\Sigma_D$ such that any right (resp. left) bracket in the sequence is closed.
Using the relations in Section~\ref{Dyck} one can check that
\begin{equation}\label{include-red}A_0\cup A_\alpha\subset B_\alpha\ \text{ and }\ A_0\cup A_\beta\subset B_\beta.\end{equation}

Define $\phi_\alpha\colon B_\alpha\to \Sigma_\alpha$ by
\[
(\phi_\alpha(\omega))_i=\begin{cases}
  \beta&\text{ if }\omega_i\in\{\beta_1,\ldots,\beta_M\},\\
  \omega_i&\text{ otherwise.}
\end{cases}\]
In other words, $\phi_\alpha(\omega)$ is obtained by replacing all $\beta_k$, $k\in\{1,\ldots,M\}$ in $\omega$ by $\beta$. Clearly $\phi_\alpha$ is continuous.
Similarly, define $\phi_\beta\colon B_\beta\to \Sigma_\beta$ by
\[(\phi_\beta(\omega))_i=\begin{cases}
    \alpha&\text{ if }\omega_i\in\{\alpha_1,\ldots,\alpha_M\},\\
    \omega_i&\text{ otherwise.}
\end{cases}\]In other words, $\phi_\beta(\omega)$ is obtained by replacing all $\alpha_k$, $k\in\{1,\ldots,M\}$ in $\omega$ by $\alpha$.
 Clearly $\phi_\beta$ is continuous too.
 We put
\[K_\alpha=\phi_\alpha(B_\alpha)\ \text{ and }\ K_\beta=\phi_\beta(B_\beta).\]

Similarly to \eqref{H-def}, 
for each $i\in\mathbb Z$ we define $ H_{\alpha,i}\colon \Sigma_\alpha\to\mathbb Z$ by       
      \[\begin{split}H_{\alpha,i}(\zeta)&=\begin{cases}\sum_{j=0}^{i-1} \sum_{k=1}^{M}(\delta_{\alpha_k,\zeta_j}-\delta_{\beta,\zeta_j})&\text{ for }  i\geq1,\\\sum_{j=i}^{-1} \sum_{k=1}^{M}(\delta_{\beta,\zeta_j}-\delta_{\alpha_k,\zeta_j})&\text{ for } i\leq -1,\\
      0&\text{ for }i=0.\end{cases}\end{split}\]
Define $\psi_\alpha\colon K_\alpha\to D^\mathbb Z$ by
\[(\psi_\alpha(\zeta))_i=\begin{cases}
 \beta_k&\text{ if }\zeta_i=\beta,\ \zeta_{s_\alpha(i,\zeta)}=
  \alpha_k,\ k\in\{1,\ldots,M\},\\
  \zeta_i&\text{ otherwise, }
\end{cases}\]
%\[(\psi_\alpha(\zeta))_i=\begin{cases}  \alpha_k&\text{ if }\zeta_i=\alpha_k,\ k\in\{1,\ldots,M\},\\  \beta_k&\text{ if }\zeta_i=\beta,\ \zeta_{s_\alpha(i,\zeta)}=   \alpha_k,\ k\in\{1,\ldots,M\}, \end{cases}\]
where \[s_\alpha(i,\zeta)=\max\{j<i+1\colon  H_{\alpha,j}(\zeta)= H_{\alpha,i+1}(\zeta)\}.\]
Clearly $\psi_\alpha$ is continuous.
Similarly, for each $i\in\mathbb Z$ we define 
$H_{\beta,i}\colon \Sigma_\beta\to\mathbb Z$ by 
\[\begin{split}
            H_{\beta,i}(\zeta)&=\begin{cases}\sum_{j=0}^{i-1} \sum_{k=1}^{M}(\delta_{\alpha,\zeta_j}-\delta_{\beta_k,\zeta_j})&\text{ for }  i\geq1,\\\sum_{j=i}^{-1} \sum_{k=1}^{M}(\delta_{\beta_k,\zeta_j}-\delta_{\alpha,\zeta_j})&\text{ for } i\leq -1,\\
      0&\text{ for }i=0.\end{cases}\end{split}\]
Define $\psi_\beta\colon K_\beta\to B_\beta$ by
%\[(\psi_\beta(\zeta))_i=\begin{cases}   \alpha_k&\text{ if }\zeta_i=\alpha,\ \zeta_{s_\beta(i,\zeta)}=\beta_k,\ k\in\{1,\ldots,M\},\\    \beta_k&\text{ if }\zeta_i=\beta_k,\ k\in\{1,\ldots,M\},\end{cases}\]
\[(\psi_\beta(\zeta))_i=\begin{cases}
   \alpha_k&\text{ if }\zeta_i=\alpha,\ \zeta_{s_\beta(i,\zeta)}=\beta_k,\ k\in\{1,\ldots,M\},\\
    \zeta_i&\text{ otherwise, }
\end{cases}\]
where \[s_\beta(i,\zeta)=\min\{j>i\colon  H_{\beta,j}(\zeta)= H_{\beta,i}(\zeta)\}.\] 
Clearly $\psi_\beta$ is continuous too.

Let us summarize the objects introduced above in the following diagram:
 \[\begin{split}
 %\label{diagram}
  \xymatrix{
    & B_\alpha\cup   B_\beta \ar@<-0.8ex>[ld]_{\phi_\alpha } 
 \ar@<0.8ex>[rd]^{\phi_\beta } &  \\
 \Sigma_\alpha\supsetneq
 K_\alpha \ar[ru]_{\psi_\alpha}\ \ \   &     &  \ \ \   
 \ar[lu]^{\psi_\beta} 
 K_\beta\subsetneq
 \Sigma_\beta.}\end{split}\]
The map $\phi_\alpha$ (resp. $\phi_\beta)$ cannot be onto $\Sigma_\alpha$ (resp. $\Sigma_\beta$), since it is a homeomorphism (see Lemma~\ref{include-lem} below) and there is no proper embedding of the full shift on $(M+1)$-symbols to the Dyck shift
\cite{HI05}. The next two lemmas were proved in \cite[Section~4]{Kri74}, but we include proofs here for completeness.

\begin{lemma}\label{include-lem}
For each $\gamma\in\{\alpha,\beta\}$ the following hold:
\begin{itemize}
\item[(a)] 
 $\phi_\gamma$ is a homeomorphism between $B_\gamma$ and $K_\gamma$ whose inverse is $\psi_\gamma$.
\item[(b)] 
$\phi_\gamma\circ\sigma|_{B_\gamma}=\sigma_\gamma\circ\phi_\gamma$ and $\sigma^{-1}\circ\psi_\gamma=\psi_\gamma\circ\sigma_\gamma^{-1}|_{K_\gamma}$.
\end{itemize}
\end{lemma}

%we use the next lemma, which was implicitly used in \cite{Kri74} and made clear in the proof of \cite[Lemma~3.2]{TY23}.
\begin{proof}
It is straightforward to check that 
$\psi_\gamma\circ\phi_\gamma(\omega)=\omega$ for every $\omega\in B_\gamma$, and
$\phi_\gamma\circ\psi_\gamma(\zeta)=\zeta$
for every $\zeta\in K_\gamma$, which verifies (a).
It is also straightforward to check (b). \end{proof}

\begin{lemma}\label{gyak-lem}
The following statements hold:
\begin{itemize}
    \item[(a)] If $\nu\in M^{\rm e}(\Sigma_\alpha,\sigma_\alpha)$ and $\nu(\Sigma_{\alpha}(0;\beta))<\frac{1}{2}$ then $\nu(K_\alpha )=1$.
    \item[(b)] If $\nu\in M^{\rm e}(\Sigma_\beta,\sigma_\beta)$ and $\nu(\Sigma_{\beta}(0;\alpha))<\frac{1}{2}$ then $\nu(K_\beta )=1$.\end{itemize}
\end{lemma}
 \begin{proof}
By the definition of $B_\alpha$ and that of $\phi_\alpha$,
we have
\[K_\alpha=\bigcap_{i=-\infty}^\infty\bigcup_{k=1}^M\left(\Sigma_{\alpha}(i;\alpha_k)\cup\left(\Sigma_{\alpha}(i;\beta)\cap\bigcup_{j=1}^\infty\{ H_{\alpha,i-j+1}=H_{\alpha,i+1}\}\right)\right),\]
and by De Morgan's laws,
\[K_\alpha^c=\bigcup_{i=-\infty}^\infty\bigcap_{k=1}^M\left(\Sigma_{\alpha}(i;\alpha_k)^c\cap\left(\Sigma_{\alpha}(i;\beta)^c\cup\bigcap_{j=1}^\infty\{ H_{\alpha,i-j+1}= H_{\alpha,i+1}\}^c\right)\right),\]
where the upper indices $c$ denote the complements in $\Sigma_\alpha$.
For each $\zeta\in \Sigma_\alpha\setminus
K_\alpha$ 
there exists $i\in\mathbb Z$ such that
$\zeta_i=\beta$ and
$ H_{\alpha,i-j+1}(\zeta)\neq H_{\alpha,i+1}(\zeta)$ for every $j\geq1$.
%Then we have $\zeta_{i-1}=\beta$, for otherwise  $ H_{\alpha,i-1}(\zeta)=H_{\alpha,i+1}(\zeta)$. If $\zeta_{i-2}\neq\beta$ then we have $\zeta_{i-3}=\beta$, for otherwise $H_{\alpha,i-3}(\zeta)= H_{\alpha,i+1}(\zeta)$. Iterating this argument gives
By induction, for every $j\geq1$ we have
\[\#\{m\in\{i-j+1,\ldots,i\}\colon \zeta_m=\beta\}>\#\{m\in\{i-j+1,\ldots,i\}\colon \zeta_m\neq\beta\}.\]
%\[\liminf_{j\to\infty}\frac{1}{j}\cdot\#\{m\in\{i-j+1,\ldots,i\}\colon \zeta_m=\beta\}\geq\frac{1}{2}.\]
If $\nu\in M^{\rm e}(\Sigma_\alpha,\sigma_\alpha)=M^{\rm e}(\Sigma_\alpha,\sigma_\alpha^{-1})$ and $\nu(\Sigma_{\alpha}(0;\beta))<\frac{1}{2}$, then
Birkhoff's ergodic theorem applied to $(\sigma_\alpha^{-1},\nu)$ yields $\nu(K_\alpha)=1$, as required in (a). A proof of (b) is analogous. \end{proof}

\subsection{Symbolic coding}\label{code-sec}
Let $(a,b)\in\Delta$. We consider the following maximal invariant sets
\[\Lambda_a=\bigcap_{n=0}^{\infty}f_a^{-n}\left(\bigcup_{\gamma\in D}{\rm int}( \Omega_\gamma^+)\right)\quad\text{ and }\quad \Lambda_{a,b}=\bigcap_{n=-\infty}^{\infty}f_{a,b}^{-n}\left(\bigcup_{\gamma\in D}{\rm int}(\Omega_\gamma)\right).\]
Define {\it coding maps} $\pi_{a}\colon (x_u,x_c)\in \Lambda_a\mapsto
(\omega_n)_{n=0}^\infty\in D^{\mathbb N_0}$ 
by \[(x_u,x_c)\in 
\bigcap_{n=0 }^\infty f_a^{-n}({\rm int}(
\Omega_{\omega_n}^+)),\]
 and $\pi_{a,b}\colon (x_u,x_c,x_s)\in \Lambda_{a,b}\mapsto
 (\omega_n)_{n=-\infty}^\infty\in D^{\mathbb Z}$ by
  \[(x_u,x_c,x_s)\in\bigcap_{n=-\infty }^\infty f_{a,b}^{-n}({\rm int}
  (\Omega_{\omega_n})).\]
   The coding map $\pi_a$ (resp. $\pi_{a,b}$) is a semiconjugacy between $f_a|_{\Lambda_a}$ (resp. $f_{a,b}|_{\Lambda_{a,b}}$) and the left shifts acting on the coding space $\overline{ \pi_{a}(\Lambda_a)}$ (resp. $\overline{\pi_{a,b}(\Lambda_{a,b})}$).
   The coding spaces are subshifts on $2M$ symbols that are independent of $(a,b)$.

The next connection between the heterochaos baker maps and the Dyck shift underlies the proofs of the main results of this paper.

   \begin{theorem}[\cite{TY23}, Theorem~1.1]\label{TY-thm}
For any $(a,b)\in \Delta$
we have \[\overline{\pi_{a}(\Lambda_{a})}=\Sigma_{D}^+
   \ \text{ and } \
   \overline{\pi_{a,b}(\Lambda_{a,b})}=\Sigma_{D}.\]
\end{theorem}

 \begin{remark}
    If $k\in\{1,\ldots,M-1\}$, then
    $f({\rm int}(\Omega_{\alpha_k}))\cap {\rm int}(\Omega_{\beta_{k+1}})=\emptyset$ and $f(\Omega_{\alpha_k})\cap \Omega_{\beta_{k+1}}\neq\emptyset$, 
   whereas the word $\alpha_k\beta_{k+1}$ is not admissible. This is the reason why the interiors have been taken in the definitions of the coding maps.\end{remark}

{\bf Convention.}
   For the rest of Section~3 we will state quite a few lemmas on the heterochaos baker maps that hold for any $(a,b)\in\Delta$. For simplicity,  we will omit the phrase `for any $(a,b)\in\Delta$' from all the statements. 

 \subsection{Preimage of the coding map}\label{str-code}
%The coding maps introduced in \S\ref{code-sec} are not injective. 
In order to clarify where the preimage of the coding map $\pi\colon\Lambda\to\Sigma_D$ is a singleton, we  
consider the set
\[A_{\alpha,\beta}=\left\{\omega\in\Sigma_D\colon \liminf_{i\to\infty}H_i(\omega)=-\infty\ \text{ or } \ \liminf_{i\to-\infty}H_i(\omega)=-\infty\right\}.\]
Note that $A_\alpha\cup A_\beta\subset A_{\alpha,\beta}$. %$A_{\alpha,\beta}\cap A_{\beta,\alpha}=\emptyset$ and $A_{\beta,\alpha}\subset\Sigma_D$.

 Let
 $\omega\in\Sigma_D$.
 For each $i\in\mathbb Z$ define 
      \[R_i(\omega)=\begin{cases}
      \bigcap_{j=0}^{i-1}f^{-j}(\Omega_{\omega_j})&\text{ for }  i\geq1,\\
      \bigcap_{j=-i+1}^{0}f^{-j}(\Omega_{\omega_j})&\text{ for } i\leq -1,\\
      [0,1]^3&\text{ for }i=0.\end{cases}\]
 % For $n\geq1$ put \[K_n^+(\omega)=\bigcap_{j=0}^{n-1}f^{-j}(\Omega_{\omega_j})\ \text{ and } \ K_n^-(\omega)=\bigcap_{j=-n+1}^{0}f^{-j}(\Omega_{\omega_j}),\] and $K_n(\omega)=K_n^+(\omega)\cap K_n^-(\omega)$. 
 See \textsc{Figure}~\ref{K-set}.
 Clearly we have $\pi^{-1}(\omega)\subset\bigcap_{i=-\infty }^\infty R_i(\omega)$.

\begin{lemma}[c.f. \cite{TY23} Lemma~3.1]\label{factor-prop}
%\textcolor{green}{Let $(a,b)\in \Delta$.}
 If $\omega\in A_{\alpha,\beta}$ then
$\bigcap_{i=-\infty }^\infty R_i(\omega)$ is a singleton.
If moreover $\omega\in\pi(\Lambda)$, then
 $\pi^{-1}(\omega)$ is a singleton.
 \end{lemma}

\begin{figure}
\centering
\includegraphics
[height=5.5cm,width=12cm]{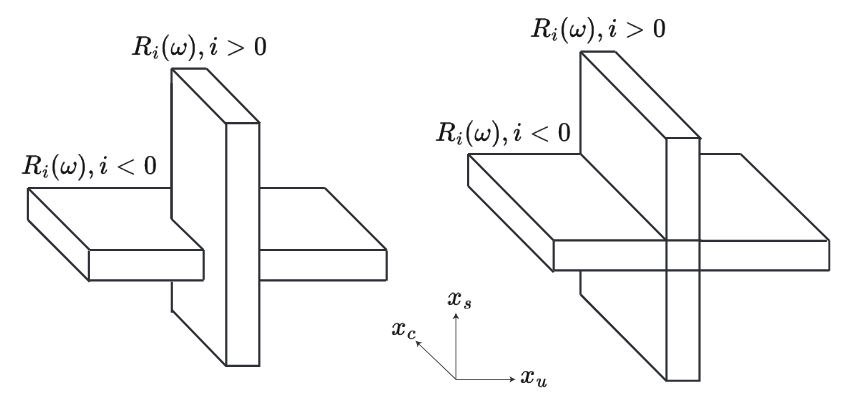}
%{K.eps}
\caption{The sets $R_i(\omega)$ in two cases: $\omega\in A_{\alpha,\beta}$ (left); 
$\omega_j\in\{\alpha_1,\ldots,\alpha_M\}$  for every $j\geq0$ and $\omega_j\in\{\beta_1,\ldots,\beta_M\}$ for every $j\leq -1$ (right).}
\label{K-set}
\end{figure}
\begin{proof} %The assertions were proved in \cite[Lemma~3.1]{TY23} by way of contradiction. Here we give a more direct geometric proof.
Suppose $\omega\in A_{\alpha,\beta}$. 
 It suffices to
show that ${\rm diam}(R_i(\omega))\to0$ as $|i|\to\infty$ where ${\rm diam}(\cdot)$ denotes the Euclidean diameter. 
For $t=u,c,s$
and two points $x=(x_u,x_c,x_s)$, $y=(y_u,y_c,y_s)$ in $[0,1]^3$,
  write
$|x-y|_t=|x_t-y_t|$.
For a subset $S$ of $[0,1]^3$ 
write ${\rm diam}^t(S)=\sup_{x,y\in S} |x-y|_t$. 
 Since the $x_u$-direction (resp. the $x_s$-direction) is uniformly expanding (resp. contracting) under the forward iteration of $f$, we have ${\rm diam}^u(R_i(\omega))\to0$ as $i\to\infty$ (resp. ${\rm diam}^s(R_i(\omega))\to0$ as $i\to-\infty$). 
%Since the $x_s$-direction is uniformly contracting under the forward iteration of $f$, we have ${\rm diam}^s(R_i(\omega))\to0$ as $i\to-\infty$.
Since $\omega\in A_{\alpha,\beta}$,
either $\liminf_{i\to\infty}H_i(\omega)=-\infty$
or $\liminf_{i\to-\infty}H_i(\omega)=-\infty$ holds.
We have
${\rm diam}^c(R_i(\omega))\to0$ as $i\to\infty$ in the first case, and 
${\rm diam}^c(R_i(\omega))\to0$ as $i\to-\infty$ in the second case. 
 We have verified the first assertion of Lemma~\ref{factor-prop}. The second one is obvious.
\end{proof}

\begin{remark}The proof of Lemma~\ref{factor-prop} shows that if $\pi^{-1}(\omega)$ is not a singleton then it is a segment parallel to the $x_c$-axis. This occurs, for example, if 
$\omega_j\in\{\alpha_1,\ldots,\alpha_M\}$  for every $j\geq0$ and $\omega_j\in\{\beta_1,\ldots,\beta_M\}$ for every $j\leq -1$. See \textsc{Figure}~\ref{K-set} right.
\end{remark}

  \subsection{The ergodic MMEs for the heterochaos baker maps}\label{str-mme} 
Let $\xi_\alpha$, $\xi_\beta$ denote the 
  Bernoulli measures on 
$\Sigma_{\alpha}$,
$\Sigma_{\beta}$ respectively associated with the probability vector $(\frac{1}{M+1},\ldots,\frac{1}{M+1})$.
For each $\gamma\in\{\alpha,\beta\}$, Lemma~\ref{gyak-lem} gives
$\xi_\gamma(K_\gamma)=1$, and
the map $\psi_\gamma\colon K_\gamma\to B_\gamma$ in Lemma~\ref{include-lem} is continuous.
Hence, the measures 
 \begin{equation}\label{ergmme}\nu_{\alpha}=\xi_\alpha\circ\psi_\alpha^{-1}\ \text{ and } \ \nu_\beta=\xi_\beta\circ\psi_\beta^{-1}\end{equation}
are $(\frac{1}{M+1},\ldots,\frac{1}{M+1})$-Bernoulli 
of entropy $\log(M+1)$, and 
charge any non-empty open subset of $\Sigma_D$.
From \eqref{include-red} and Lemma~\ref{trichotomy}
 it follows that \begin{equation}\label{Kri-eq}M_{\rm max}^{\rm e}(\Sigma_D,\sigma)=\{\nu_\alpha, \nu_\beta\},\end{equation}
 which proves Theorem~\ref{k-thm}.
 
A direct calculation based on \eqref{ergmme} gives
\begin{equation}\label{balance}\nu_\alpha(\Sigma_{D}(0;\alpha_k))=\nu_\beta(\Sigma_{D}(0;\beta_k))=\frac{1}{M+1}\ \text{ for  }k=1,\ldots,M.\end{equation}
The ergodicity of $\nu_\alpha$, $\nu_\beta$ and \eqref{balance} together imply
 \begin{equation}\label{measure-nu}\nu_\alpha(A_\alpha )=1\ \text{ and }\ 
 \nu_\beta(A_\beta )=1.\end{equation}
By Lemma~\ref{factor-prop} and \eqref{measure-nu}, $\pi^{-1}(\omega)$ is a singleton for $\nu_\alpha$-a.e. $\omega$ and for $\nu_\beta$-a.e. $\omega$.
 Moreover, by \cite[Lemma~1]{TY23-c} we have \[\nu_\alpha(\pi(\Lambda))=1\ \text{ and }\ \nu_\beta(\pi(\Lambda))=1.\] 
 Hence, we can define two Borel probability measures
\begin{equation}\label{mualpha}\mu_\alpha=\nu_\alpha\circ\pi\ \text{ and }\ 
 \mu_\beta=\nu_\beta\circ\pi,\end{equation}
 that are $f$-invariant of entropy $\log(M+1)$ and ergodic. Later in Section~\ref{pfthmab-sec} we will show that the ergodic MMEs of $f$ are precisely $\mu_\alpha$ and $\mu_\beta$.

\subsection{Lyapunov exponents of invariant measures}\label{lyapunov-sec}
 We would like to define Lyapunov exponents 
for $f$-invariant measures directly, but $f$ is not continuous.
To dispel technicalities we first define `Lyapunov exponents' for shift invariant measures on $\Sigma_D$, and then pull them back to define Lyapunov exponents for $f$-invariant measures.
With a slight abuse of notation, write \[M^{\rm e}(\Lambda,f)=\{\mu\in M^{\rm e}([0,1]^3,f)\colon\mu(\Lambda)=1\}.\]

     Let $(a,b)\in \Delta$.
   For $t=u,c$ define $\varphi^t\colon \Sigma_{D}\to\mathbb R$ by
   \[\begin{split}\varphi^u(\omega)&=\begin{cases}-\log a&\text{ on }  \bigcup_{k=1}^{M}\Sigma_{D}(0;\alpha_k),\\
 -\log(1-Ma)&\text{ on }  \bigcup_{k=1}^{M}\Sigma_{D}(0;\beta_k),\end{cases}\\
 \varphi^c(\omega)&=\begin{cases}-\log M&\text{ on }  \bigcup_{k=1}^{M}\Sigma_{D}(0;\alpha_k),\\
 \log M&\text{ on }  \bigcup_{k=1}^{M}\Sigma_{D}(0;\beta_k).\end{cases}\end{split}\]
% \varphi^s(p)&=\begin{cases}\log (1-mb)&\quad\text{ on }  \bigcup_{i=1}^{m}[\alpha_i],\\
% \log b&\quad\text{ on }  \bigcup_{i=1}^{m}[\beta_i],
% \end{cases}
%  Define a matrix-valued continuous map $M\colon\Sigma_{H\!C}\to {\rm GL}(3,\mathbb R)$ by
%  \begin{equation}\label{cocycle0}M(\omega)=\begin{cases}\begin{pmatrix}\frac{1}{a} & 0 & 0\\
% 0 & \frac{1}{2} & 0\\
% 0 & 0 & 1-(r-1)b\end{pmatrix}&\text{ on } \bigcup_{i=1}^{r-1}[i],\\
% \empty\\
% \begin{pmatrix}\frac{1}{1-(r-1)a} & 0 & 0\\
% 0 & 2 & 0\\
% 0 & 0 & b\end{pmatrix}&\text{ on } \bigcup_{i=r}^{2r-2}[i].\end{cases}\end{equation}
%  As a shorthand of \eqref{cocycle0} we write
%\[M(\omega)=\begin{pmatrix}\exp\varphi^u(\omega) & 0 & 0\\
% 0 & \exp\varphi^c(\omega) & 0\\
% 0 & 0 & \exp\varphi^s(\omega)\end{pmatrix}.\]
% Then for each $n\in\mathbb Z\setminus\{0\},$
%\[M^n(\theta)=\begin{pmatrix}\exp S_n\varphi^u(\theta) & 0 & 0\\ 0 & \exp S_n\varphi^c(\theta) & 0\\ 0 & 0 & \exp S_n\varphi^s(\theta)\end{pmatrix}.\]
    % Let $\sigma\colon\Sigma\to\Sigma$ be a subshift and 
% let $A\colon\Sigma\to{\rm GL}(3,\mathbb R)$ be continuous.
%For $n\geq1$, we denote \[M^n(\theta)=M(\sigma^{n-1}\theta)\cdots M(\sigma \theta)M(\theta)\quad\text{and}\quad M^{-n}(\theta)=M^{-1}(\sigma^{-n}\theta)\cdots M^{-1}(\sigma^{-1}\theta).\] 
 For a measure $\nu\in  M(\Sigma_D,\sigma)$ we put
$\chi^t(\nu)=\int\varphi^t{\rm d}\nu$.
We say $\omega\in \Sigma_{D}$ is {\it regular} if for each $t=u,c$
 there exists $\chi^t(\omega)\in\mathbb R$ such that
 \[\lim_{n\to\infty}\frac{1}{n} \sum_{i=0}^{n-1}\varphi^t(\sigma^i\omega)=
\lim_{n\to\infty}\frac{1}{n} \sum_{i=0}^{n-1}\varphi^t(\sigma^{-i}\omega) =\chi^t(\omega).\]
% there exist a decomposition
% $\mathbb R^d=\bigoplus_{j=1}^{\l(x)} E_x^j$ into linear subspaces and real numbers  $\chi_j(x)$ $j=1,\ldots,l(x)$ such that
% \[\lim_{n\to\pm\infty}\frac{1}{n}\log\|M^n(x)v\|=\chi_j(x)\]
% for all nonzero vector $v\in E_x^j$. 
%For a regular point $\omega\in \Sigma_{D}$, we call the numbers
% $\chi^u(\omega)$,  $\chi^c(\omega)$,  $\chi^s(\omega)$
%  {\it pointwise unstable, central, stable exponents} at $\omega$ respectively.
%Note that $\chi^u(\theta)>\max\{\chi^c(\theta),\chi^c(\theta)\}$.
%The set of regular points is a Borel set which depends on $(a,b)$.
% $x\mapsto \log\|A(x)\|$ and $x\mapsto \log\|A^{-1}(x)\|$
 %are $\mu$-integrable, 
  Consider the cocycle of $2\times2$ matrices over $(\Sigma_{D},\sigma)$
induced from $\varphi^u$, $\varphi^c$.
Since these functions are bounded continuous,
 by Oseledec's theorem  \cite{Ose68},
 for each $\nu\in M(\Sigma_D,\sigma)$ 
 the set of regular points has full $\nu$-measure. 
 If $\nu$ is ergodic, then $\chi^t(\omega)=\chi^t(\nu)$ $\nu$-a.e.

  Now, for $\mu\in M^{\rm e}(\Lambda,f)$ and $t=u,c$ we put \[\chi^t(\mu)=\chi^t(\mu\circ\pi^{-1}),\] and
call the numbers
 $\chi^u(\mu)$,  $\chi^c(\mu)$ {\it unstable, central Lyapunov exponents of $\mu$} respectively.
 They
equal the
exponential growth rates of the diagonal elements of the Jacobian matrices of $f$
in the $x_u,x_c$-directions respectively along $\mu$-typical orbits. When we need to make explicit the map $f$, we write $\chi^t(f,\mu)$.
A stable Lyapunov exponent corresponding to the exponential growth rate in the $x_s$-direction can also be defined, but we do not need it.

We give three examples that can be checked by a direct calculation. As in \eqref{chia}, the central Lyapunov exponent of the Lebesgue measure with respect to $g_a$ is 
\begin{equation}\label{l-c-leb}\chi^c(g_a,{\rm Leb})=(1-2Ma)\log M.\end{equation}
From \eqref{balance}, 
the central Lyapunov exponents of the two ergodic measures in \eqref{mualpha} are given by 
\begin{equation}\label{l-c-mme}\chi^c(\mu_{\alpha})=-\frac{M-1}{M+1}\log M<0\ \text{ and } \ \chi^c(\mu_{\beta})=\frac{M-1}{M+1}\log M>0.\end{equation}

To proceed, we decompose $M^{\rm e}(\Lambda,f)$ into the following three subsets:
\[\begin{split}
     M^{\rm e}_-(\Lambda,f)&=\{\mu\in M^{\rm e}(\Lambda,f)\colon\chi^c(\mu)<0\};\\
    M^{\rm e}_+(\Lambda,f)&=\{\mu\in M^{\rm e}(\Lambda,f)\colon \chi^c(\mu)>0\};\\
     M_0^{\rm e}(\Lambda,f)&=\{\mu\in M^{\rm e}(\Lambda,f)\colon \chi^c(\mu)=0\}.
    \end{split}\]

\begin{lemma}\label{class-lem}\
%\textcolor{green}{For any $(a,b)\in \Delta$ we have:}
%the following equalities hold
\begin{itemize}
\item[(a)] $M^{\rm e}_-(\Lambda,f)\subset\{\mu\in M^{\rm e}(\Lambda,f)\colon\mu\circ \pi^{-1}(A_\alpha)=1\};$
\item[(b)] $M^{\rm e}_+(\Lambda,f)\subset\{\mu\in M^{\rm e}(\Lambda,f)\colon\mu\circ \pi^{-1}(A_\beta)=1\};$
\item[(c)] $M_0^{\rm e}(\Lambda,f)\supset\{\mu\in M^{\rm e}(\Lambda,f)\colon\mu\circ \pi^{-1}(A_0)=1\}.$
\end{itemize}
\end{lemma}
\begin{proof}
By virtue of Lemma~\ref{trichotomy},
the assertion is a consequence of the definition of $H_i$, $\varphi^c$ and Birkhoff's ergodic theorem.
\end{proof}
We need a few lemmas on Lyapunov exponents. 
%The first one is on their symmetry that comes from the symmetry of the Dyck shift.
Define a homeomorphism $\iota\colon[0,1]^3\to[0,1]^3$ by
\[\iota(x_u,x_c,x_s)=(1-x_s,1-x_c,1-x_u).\]
Define $\iota_*\colon\mu\in M([0,1]^3,f_{a,b})\mapsto \mu\circ \iota^{-1}\in M([0,1]^3)$.
\begin{lemma}\label{inverse}\
%\textcolor{green}{For any $(a,b)\in \Delta$ the following statements hold:}
\begin{itemize}
\item[(a)] The map
$\iota_*$ is an entropy preserving homeomorphism
from $M([0,1]^3,f_{a,b})$ to 
$M([0,1]^3,f_{b,a})$. 

\item[(b)]
We have $\iota_*(M_-^{\rm e}(\Lambda,f_{a,b}))=M_+^{\rm e}(\Lambda,f_{b,a})$,
$\iota_*M_+^{\rm e}(\Lambda,f_{a,b})=M_-^{\rm e}(\Lambda,f_{b,a})$ and
$\iota_*M_0^{\rm e}(\Lambda,f_{a,b})=M_0^{\rm e}(\Lambda,f_{b,a})$.\end{itemize}
\end{lemma}
\begin{proof} By \cite[Lemma~3.9]{TY23}, $\iota_*$ is an entropy preserving bijection from $M([0,1]^3,f_{a,b})$ to 
$M([0,1]^3,f_{b,a})$. 
Since $\iota_*$ is a homeomorphism, (a) holds.

We define $\rho\colon D\to D$ by $\rho (\alpha_k)=\beta_k$ and $\rho(\beta_k)=\alpha_k$
for $k=1,\ldots,M$, and define a map $\iota_D\colon \Sigma_D\to D^\mathbb Z$
by $\iota_D((\omega_i)_i)=(\rho(\omega_{-i}))_i.$ As in \cite[Lemma~5.4]{T23} or in the proof of \cite[Lemma~3.8]{TY23}, $\iota_D$ induces a homeomorphism on $\Sigma_D$ which we still denote by $\iota_D$. Clearly we have
 %$\iota_D\circ\sigma=\sigma^{-1}\circ\iota_D$ and
 $\iota(\Lambda_{a,b})=\Lambda_{b,a}$ and
 $\pi_{b,a}\circ\iota|_{\Lambda_{a,b}}=\iota_D\circ\pi_{a,b}$. 

Let $\mu\in M_-^{\rm e}(\Lambda,f_{a,b})$. We have
%and let $x\in\Lambda$ be a generic point of $\mu$. Then $\pi_{b,a}(\iota(x))=\iota_D(\pi_{a,b}(x))$ is a generic point of the measure ${\pi_{b,a}}_*\circ\iota_*(\mu)={\iota_D}_*\circ{\pi_{a,b}}_*(\mu)$.
$\iota_*(\mu)\circ\pi_{b,a}^{-1}=\mu\circ(\iota_D\circ \pi_{a,b})^{-1}$.
The definition of $\iota_D$ implies 
%$\iota_*(\mu)\circ\pi_{b,a}^{-1}(A_\beta)=1$, and
%yields  
$\iota_*(\mu)\in M_+^{\rm e}(\Lambda,f_{b,a})$. We have verified that
$\iota_*(M_-^{\rm e}(\Lambda,f_{a,b}))\subset M_+^{\rm e}(\Lambda,f_{b,a}).$
The same reasoning yields
$\iota_*(M_+^{\rm e}(\Lambda,f_{a,b}))\subset M_-^{\rm e}(\Lambda,f_{b,a})$ and $\iota_*(M_0^{\rm e}(\Lambda,f_{a,b}))\subset M_0^{\rm e}(\Lambda,f_{b,a}).$
Since %$\iota_*\iota_*$ equals the identity, 
$\iota_*(M^{\rm e}(\Lambda,f_{a,b}))= M^{\rm e}(\Lambda,f_{b,a})$, all the above three inclusions are actually set equalities
as required in (b).
\end{proof}

The next lemma is a version of Ruelle's inequality \cite{Rue78}.
\begin{lemma}[\cite{TY23}, Lemma~3.5]\label{ruelle-ineq}
%\textcolor{green}{Let $(a,b)\in \Delta$ and}
For any $\mu\in M^{\rm e}(\Lambda,f)$ we have \[h(\mu)\leq\chi^u(\mu)+\max\{\chi^c(\mu),0\}.\]
\end{lemma}
The next lemma provides an upper bound on the unstable Lyapunov exponent of ergodic measures whose central Lyapunov exponent is $0$. 
 \begin{lemma}\label{zero-u}
%\textcolor{green}{ Let  $(a,b)\in \Delta$ and}
For any $\mu\in M_0^{\rm e}(\Lambda,f)$ we have \[\chi^u(\mu)\leq-\log\sqrt{a(1-Ma)}.\]
\end{lemma}
\begin{proof}
From Birkhoff's ergodic theorem
 applied to $\mu\circ\pi^{-1}$,
 there exists $\omega\in\pi(\Lambda)$ such that
\begin{equation}\label{lambda-ua}
\lim_{n\to\infty}\frac{1}{n}\sum_{i=0}^{n-1}\varphi^c(\sigma^i\omega)=0\ \text{ and }\ \lim_{n\to\infty}\frac{1}{n}\sum_{i=0}^{n-1}\varphi^u(\sigma^i\omega)=\chi^u(\mu).\end{equation}
The first equality in \eqref{lambda-ua} is equivalent to
\[\lim_{n\to\infty}\frac{1}{n}\#\{i\in\{0,\ldots,n-1\}\colon \omega_i\in\{\alpha_1,\ldots,\alpha_M\}\}=\frac{1}{2}.\]
Hence,
 for any $\varepsilon>0$ there exists $n(\varepsilon)\geq1$ such that for every $n\geq n(\varepsilon)$ we have
\[\begin{split}\frac{1}{n}\sum_{i=0}^{n-1}\varphi^u(\sigma^i\omega)&\leq\left(\frac{1}{2}+\varepsilon\right)\left(\log\frac{1}{a}+\log\frac{1}{1-Ma}\right).\end{split}\] 
 Letting $n\to\infty$ in this inequality, combining the result with the first equality in \eqref{lambda-ua} and finally decreasing $\varepsilon$ to $0$ we obtain the desired inequality.
    \end{proof}

To conclude that $\mu_\alpha$, $\mu_\beta$ are the only ergodic MMEs for $f$, we will need the next lemma.
\begin{lemma}\label{meas-entropy}\
%\textcolor{green}{For any $(a,b)\in \Delta$ the following statements hold.}

\begin{itemize}
%\item[(i)] We have $h(\mu_\alpha)=h(\mu_\beta)=\log(M+1)$. 

\item[(a)] For any 
$\mu\in(M_-^{\rm e}(\Lambda,f)\cup (M_+^{\rm e}(\Lambda,f))\setminus\{\mu_{\alpha},\mu_\beta\}$
we have $h(\mu)<\log(M+1)$.

\item[(b)] For any $\mu\in M^{\rm e}([0,1]^3,f)\setminus M^{\rm e}(\Lambda,f)$ we have $h(\mu)<\log(M+1)$.\end{itemize}
\end{lemma}
\begin{proof}
Item (a) follows from combining \eqref{Kri-eq}, Lemma~\ref{factor-prop} and Lemma~\ref{class-lem}(a), (b).
%Below we give a proof of (iii) that was not included in \cite{TY23}. 

To prove (b), let \[A_\infty=\bigcup\{{\rm supp}(\mu)\colon\mu\in M^{\rm e}([0,1]^3,f),\mu(\Lambda)=0\},\]
where ${\rm supp}(\mu)$ denotes the smallest closed subset of $[0,1]^3$ with $\mu$-measure $1$.
It is not hard to show that the coding map $\pi\colon\Lambda\to\Sigma_D$ can be extended to a Borel map
 $\tilde\pi\colon A_\infty\cup\Lambda\to\Sigma_D$
such that $\tilde\pi\circ f|_{A_\infty}=\sigma\circ \tilde\pi$ and $\tilde\pi|_{A_\infty}$ is injective.
Let $\mu\in M^{\rm e}([0,1]^3,f)\setminus M^{\rm e}(\Lambda,f)$. 
Since $\mu(A_\infty)=1$,
the measure
$\mu\circ\tilde\pi^{-1}$ is shift invariant and $h(\mu)=h(\mu\circ\tilde\pi^{-1})$.
 Since
$\pi^{-1}(\omega)\cup\tilde\pi^{-1}(\omega)\subset\bigcap_{i=-\infty }^\infty R_i(\omega)$ holds for every $\omega\in\Sigma_D$,
Lemma~\ref{factor-prop} and $\nu_\alpha(\pi(\Lambda))=1$ together imply that 
for $\nu_\alpha$-a.e. $\omega\in\Sigma_D$,
$\tilde\pi^{-1}(\omega)=\pi^{-1}(\omega)$ holds and these sets are singletons.
Since $\mu_\alpha(\Lambda)=1$ and $\mu(\Lambda)=0$
we obtain $\mu\circ\tilde\pi^{-1}\neq \nu_\alpha$.
The same argument yields
$\mu\circ\tilde\pi^{-1}\neq \nu_\beta$, and thus
 Theorem~\ref{k-thm} yields
 $h(\mu)<\log (M+1)$ as required in (b).
 \end{proof}

\subsection{Correspondence of invariant measures}\label{projection-sec}
The next lemma clarifies a correspondence between invariant measures of $f_a$ and that of $f_{a,b}$. Recall that $p\colon[0,1]^3\to[0,1]^2$ denotes the canonical projection.
Define $p_*\colon M([0,1]^3)\to M([0,1]^2)$ by $p_*(\mu)=\mu\circ p^{-1}.$
For simplicity we denote $p_*(\mu)$ by $\mu^+$.

\begin{lemma}\label{meas-proj}\
%\textcolor{green}{For any $(a,b)\in \Delta$ the following statements hold.}
\begin{itemize}

\item[(a)] For any $\mu\in M([0,1]^3,f_{a,b})$ we have
 $h(f_{a,b},\mu)=h(f_a,\mu^+)$. Moreover,
 $\mu^+$ is ergodic if and only if $\mu$ ergodic.
\item[(b)] 
For any $\mu\in M([0,1]^2,f_a)$ there exists a unique element $\hat\mu$ of $M([0,1]^3,f_{a,b})$
such that $(\hat\mu)^+=\mu$.
\end{itemize}
\end{lemma}
%If $A$ is Borel, $@(A)$ is Borel. $\{A\colon \text{$@(A)$ is Borel}\}$ contains all open sets. 
%Since $@$ is a Borel map, for any $\nu\in\mathcal M(f_a)$ we can define a measure $\hat\nu\in\mathcal M(f_{a,b})$ by $\hat\nu=\nu\circ @$.  (invariance: $\nu\circ@(A)=\nu\circ@(f_{a,b}^{-1}(A)=\nu\circ f_a^{-1}@(A)$?) (false: $@$ and $f_{a,b}^{-1}$ are not commutative ($@f_{a,b}^{-1}=f_a^{-1}@$ on $\{\chi^c(x)=0\}$), so we do not know if $\hat\nu$ is invariant.) Then $(\hat\nu)^+=\nu$ holds. So, $\heartsuit$ is surjective.

%\begin{remark}The uniqueness in (b) will not be needed in this paper.\end{remark}
\begin{proof}
We start with introducing the standard notation in \cite{Wal82}.
Let $X$ be a measurable space and let $T\colon X\to X$ be a measurable map. 
We assume all partitions of $X$ to be finite partitions into measurable sets.
For two partitions $\mathcal P$, $\mathcal Q$ of $X$, let $\mathcal P\vee\mathcal Q$ denote their join: $\mathcal P\vee\mathcal Q=\{P\cap Q\colon P\in\mathcal P,\ Q\in\mathcal Q\}$.
The join of partitions $\mathcal P$, $T^{-1}\mathcal P,\ldots,T^{-n+1}\mathcal P$ $(n\geq1)$ 
is denoted by $\bigvee _{i=0}^{n-1} T^{-i}\mathcal P$.
If $T$ has a measurable inverse, we define $\bigvee _{i=-n+1}^{n-1} T^{-i}\mathcal P$ analogously.

Let $(a,b)\in\Delta$ and
let $\mu\in M([0,1]^3,f_{a,b})$. It is easy to see that there exists a Borel subset $R$ of $[0,1]^3$ such that the restriction of $f_{a,b}$ to $R$ has a measurable inverse and $\mu(R)=1$.
Take a finite generator $\mathcal Q$ of $f_a$ \cite{Kri70}, which is a finite partition of $[0,1]^2$.
%: $\xi$ is a finite partition of $[0,1]^2$ into Borel sets such that the smallest sigma-algebra containing  $\{f_a^{-i}(A)\colon A\in\xi,i\in\{0,\ldots,n-1\}$ is the Borel sigma-algebra.
%\textcolor{red}{for any relatively open set $U\subset[0,1]^2$ and any $x\in U$ there are $n$ and  $B\in\bigvee_{i=0}^{n-1}f_{a}^{-i}\xi$ such that $x\in B\subset U$.}
From the uniform contraction of $f_{a,b}$ in the $x_s$-direction it follows that $\hat{\mathcal Q}=\{(Q\times[0,1])\cap R\colon Q\in\mathcal Q\}$ is a generator of $f_{a,b}|_R$.
For $x\in R$ and $n\geq1$
let $\mathcal Q_n(x)$ denote the element of
$\bigvee _{i=0}^{n-1} f_a^{-i}\mathcal Q$ that contains $p(x)$, and 
let $\hat{\mathcal Q}_n(x)$ denote the element of
$\bigvee _{i=-n+1}^{n-1} (f_{a,b}|_R)^{-i}\hat{\mathcal Q}$ that contains $x$.
We have
$\mu(\hat{\mathcal Q}_n(x))=\mu^+(\mathcal Q_n(x)).$
By Shannon-McMillan-Breiman's theorem,
$\lim_{n\to\infty}\frac{1}{n}\log\mu(\hat{\mathcal Q}_n(x))$ exists for $\mu$-a.e. $x\in R$ and we obtain
\[\begin{split}h(\mu)&=-\int\lim_{n\to\infty}\frac{1}{n}\log\mu(\hat{\mathcal Q}_n(x)){\rm d}\mu(x)\\
&=-\int
\lim_{n\to\infty}\frac{1}{n}\log\mu^+(\mathcal Q_n(x)){\rm d}\mu(x)=h(\mu^+),\end{split}\]
as required in (a).

Since $f_a\circ p=p\circ f_{a,b}$,
 $p_*$ maps
$M^{\rm e}([0,1]^3,f_{a,b})$ to
$M^{\rm e}([0,1]^2,f_{a})$.
For each $\mu\in M([0,1]^3,f_{a,b})$
let $\mu=\int_{M^{\rm e}([0,1]^3,f_{a,b})} \nu{\rm d}\eta(\nu)$ denote its ergodic decomposition, where $\eta$
is a Borel probability measure on $M^{\rm e}([0,1]^3,f_{a,b})$. Then
$\mu^+=\int_{M^{\rm e}([0,1]^2,f_{a})} \nu{\rm d}\eta^+(\nu)$, $\eta^+=\eta\circ(p_*|_{M^{\rm e}([0,1]^3,f_{a,b})})^{-1}$  is the ergodic decomposition of $\mu^+$. So the last assertion of (a) holds.

In what follows we write $f$ for $f_{a,b}$, and adapt the argument in the proof of \cite[Lemma~1.13]{Bow75}. 
For $d\in\{2,3\}$ let $C([0,1]^d)$ denote the space of real-valued continuous functions on $[0,1]^d$ endowed with the supremum norm. 
%$\|\phi\|=\sup\{|\phi(x)|\colon x\in [0,1]^d\}$.
%Let $C^*([0,1]^d)$ denote the dual of $C([0,1]^d)$.
For $\phi\in C([0,1]^3)$
and $n\geq1$ define
\[{\rm var}_n\phi=\sup_{x\in[0,1]^2}\sup_{y,z\in f^n(p^{-1}(x))}|\phi(y)-\phi(z)| .\]
Since the $x_s$-direction is uniformly contracting under the forward iteration of $f$ and $\phi$ is uniformly continuous, we have ${\rm var}_n\phi\to0$ as $n\to\infty$.
Define $\phi^*\in C([0,1]^2)$ by
\[\phi^*(x)=\min\{\phi(y)\colon y\in p^{-1}(x)\}.\]
For $m,n\geq 0$ and $x\in[0,1]^3$ we have
\[|(\phi\circ f^n)^*\circ f^m(x)-(\phi\circ f^{n+m})^*(x)|\leq {\rm var}_n\phi,\]
and hence
\[\begin{split}\left|\mu((\phi\circ f^n)^*)-\mu((\phi\circ f^{n+m})^*)\right|&=\left|\mu((\phi\circ f^n)^*\circ f^m)-\mu((\phi\circ f^{n+m})^*)\right|\leq{\rm }{\rm var}_n\phi,\end{split}\]
where $\mu(\phi)$ stands for $\int\phi {\rm d}\mu$. 
Hence, $\hat\mu(\phi)=\lim_{n\to\infty}\mu((\phi\circ f^n)^*)$ exists by Cauchy's criterion. It is straightforward to check that
$\hat\mu$ belongs to the dual of $C([0,1]^3)$.
%$\hat\mu\in C^*([0,1]^3)$. 
By Riesz's representation theorem, $\hat\mu$ defines an element of $M([0,1]^3)$ which we still denote by $\hat\mu$ despite the possible ambiguity.
Note that for any $\phi\in C([0,1]^3)$,
\[\hat\mu(\phi\circ f)=\lim_{n\to\infty}\mu(\phi\circ f^{n+1})=\hat\mu(\phi),\]
proving that $\hat\mu$ is $f$-invariant.

For each $\varphi\in C([0,1]^2)$
define $\hat\varphi\in C([0,1]^3)$
by $\hat\varphi(x_u,x_c,x_s)=\varphi(x_u,x_c)$.
Since $(\hat\varphi)^*=\varphi$ 
we have
$(\hat\mu)^+(\varphi)=\hat\mu(\hat\varphi)=\lim_{n\to\infty}\mu((\hat\varphi\circ f^n)^*)=\lim_{n\to\infty}\mu(\varphi\circ f_a^n)=\mu(\varphi)$, and hence
 $(\hat\mu)^+=\mu$ by Riesz's representation theorem.

It is left to show that if $\mu,\nu\in M([0,1]^3,f)$ are distinct then  $\mu^+\neq\nu^+$.
 Since Lipschitz continuous functions are dense in $C([0,1]^3)$, there is 
 $\varphi\in C([0,1]^3)$ that is Lipschitz continuous and satisfies $\mu(\varphi)\neq\nu(\varphi)$. 
 %By the argument in the proof of \cite[Lemma~1.6]{Bow75}, $\varphi$ is homologous to a continuous function $\psi\colon[0,1]^3\to\mathbb R$ which is constant on the  sets of the form  $\{x\}\times[0,1]^2$, $x\in[0,1]$. 
Adapting the argument in the proof of \cite[Lemma~1.6]{Bow75},
 from $\varphi$ we construct a Borel function $\psi$ on $[0,1]^3$
 such that $\mu^+(\psi)\neq\nu^+(\psi)$.

 Define $r\colon[0,1]^3\to[0,1]^3$ by $r(x)=(x_u,x_c,1)$ for $x=(x_u,x_c,x_s)$.
 Define $u\colon[0,1]^3\to\mathbb R$ by
 \[u(x)=\sum_{j=0}^{\infty}(\varphi(f^j(x))-\varphi(f^jr(x))).\]
 Since the $x_s$-direction is contracting by factor $b$ or $1-Mb$ under the forward iteration of $f$, for any $x\in[0,1]^3$ and any $j\geq1$ we have 
\[|\varphi(f^j(x))-\varphi(f^jr(x))|\leq {\rm Lip}(\varphi)(\max\{b,1-Mb\})^j,\]
where ${\rm Lip}(\varphi)$ denotes the Lipschitz constant of $\varphi$.
Hence
$u$ is well-defined and continuous.
For the Borel function $\psi=\varphi+u\circ f-u$
we have
\[\begin{split}\psi(x)&=\varphi(x)+\sum_{j=0}^\infty(\varphi(f^j(f(x)))-\varphi(f^jr(f(x))))-\sum_{j=0}^\infty(\varphi(f^j(x))-\varphi(f^jr(x)))\\
&=\sum_{j=0}^\infty(\varphi(f^jr(x))-\varphi(f^jr(f(x))).
\end{split}\]
The last expression depends only on $x_u$ and $x_c$.
Therefore, $p(x)=p(y)$ implies
 $\psi(x)=\psi(y)$.
In other words, $\psi$ may be viewed as a function on $[0,1]^2$, and so $\mu(\varphi)=\mu(\psi)=\mu^+(\psi)$ and $\nu(\varphi)=\nu(\psi)=\nu^+(\psi)$. Hence we obtain
$\mu^+(\psi)\neq\nu^+(\psi)$.
Approximating $\psi$ in the $L^1$ norm, we can find $\psi_0\in C([0,1]^3)$ such that
$\mu^+(\psi_0)\neq\nu^+(\psi_0)$.
 Riesz's representation theorem
yields $\mu^+\neq\nu^+$.
 This completes the proof of (b).
\end{proof}

\subsection{When the Lebesgue measure is preserved}\label{leb-section}
For completeness we include a proof of the next elementary lemma. 
\begin{lemma}\label{preserve}\
%\textcolor{green}{For any $(a,b)\in\Delta$, the following statements hold:}
\begin{itemize}
\item[(a)] $f_a$ preserves the Lebesgue measure on $[0,1]^2$.
\item[(b)] $f_{a,b}$ preserves the Lebesgue measure on $[0,1]^3$ if and only if $a+b=\frac{1}{M}$.
\end{itemize}
\end{lemma}
\begin{proof}
Let $|$ $\cdot$ $|$ denote the Lebesgue measure on $[0,1]^d$.
Let $a\in(0,\frac{1}{M})$ and
let $S$ be a Borel subset of $[0,1]^2$.
From the definition of $f=f_a$, for every $k\in\{1,\ldots,M\}$ we have
\[|f^{-1}(S\cap f(\Omega_{\alpha_k}^+))|=
Ma|S\cap f(\Omega_{\alpha_k}^+)|,\]
\[ |f^{-1}(S\cap f(\Omega_{\beta_k}^+))|=\frac{1-Ma}{M}|S\cap f(\Omega_{\beta_k}^+)|,\]
and
\[|S\cap f(\Omega_{\beta_k}^+)|=|S|.\]
We also have
\[\sum_{k=1}^M|S\cap f(\Omega_{\alpha_k}^+)|=|S|.\]
Using these four equalities we obtain
\[\begin{split}|f^{-1}(S)|&=\sum_{k=1}^M|f^{-1}(S\cap f(\Omega_{\alpha_k}^+))|+\sum_{k=1}^M|f^{-1}(S\cap f(\Omega_{\beta_k}^+))|\\
&=Ma\sum_{k=1}^M|S\cap f(\Omega_{\alpha_k}^+)|+\frac{1-Ma}{M}\sum_{k=1}^M|S\cap f(\Omega_{\beta_k}^+)|\\
&=\sum_{\gamma\in D}|S\cap f(\Omega_{\gamma}^+)|=|S|,\end{split}\]
as required in (a).

Let $(a,b)\in\Delta$. By the definition of $f=f_{a,b}$, clearly we have $|\Omega_{\beta_1}|=\frac{1-Ma}{M}$ and
$|f(\Omega_{\beta_1})|=b$.
If ${\rm Leb}$ is $f$-invariant then
these must be equal, namely $a+b=\frac{1}{M}$.

Conversely, suppose $a+b=\frac{1}{M}$. Let $S$ be a Borel subset of $[0,1]^3$. By the definition of $f=f_{a,b}$,
for every $k\in\{1,\ldots,M\}$ we have
\[
|f^{-1}(S\cap f(\Omega_{\alpha_k}))|=\frac{Ma}{1-Mb}|S\cap f(\Omega_{\alpha_k})|\]
and
\[|f^{-1}(S\cap f(\Omega_{\beta_k}))|=\frac{1-Ma}{Mb}|S\cap f(\Omega_{\beta_k})|.\]
Using these two equalities and $a+b=\frac{1}{M}$ we obtain
\[\begin{split}|f^{-1}(S)|&=\sum_{k=1}^M|f^{-1}(S\cap f(\Omega_{\alpha_k}))|+\sum_{k=1}^M|f^{-1}(S\cap f(\Omega_{\beta_k}))|\\
&=\frac{Ma}{1-Mb}\sum_{k=1}^M|S\cap f(\Omega_{\alpha_k})|+\frac{1-Ma}{Mb}\sum_{k=1}^M|S\cap f(\Omega_{\beta_k})|\\
&=\sum_{\gamma\in D}|S\cap f(\Omega_{\gamma})|=|S|,\end{split}\]
as required in (b).
\end{proof}

\section{Proofs of the main results}
In this last section we prove Theorems~\ref{thm-a}, \ref{thm-c}, \ref{Bernoulli-thm}
using the results in Section~3.
We first prove Theorem~\ref{Bernoulli-thm} in Section~\ref{pfthm2.8}. We then prove Theorem~\ref{thm-a} in 
Section~\ref{pfthmab-sec}. Finally
we prove Theorem~\ref{thm-c} in Section~\ref{pfthmc-sec}.

\subsection{Proof of Theorem~\ref{Bernoulli-thm}(c)}\label{pfthm2.8}
 We only give a proof of the first statement  since that of the second one is analogous. Let $a\in (0,\frac{1}{2M})$. Recall that $g_a=f_{a,\frac{1}{M}-a}$ as in \eqref{ga}.
By \eqref{l-c-leb} we have $\chi^c(g_a,{\rm Leb})>0$. Write
$\pi$ for the coding map $\pi_{a,\frac{1}{M}-a}$.

Since the collection $\{\sigma^i(\Sigma_\beta(0;\xi))\colon i\in\mathbb Z,\ \xi\in L(\Sigma_\beta)\setminus\{\emptyset\}\}$
 of cylinders generates the Borel sigma-algebra of $\Sigma_\beta$, in order to show the coincidence of two shift invariant measures on $\Sigma_\beta$ it suffices to 
show that they give equal weights to $\Sigma_\beta(0;\xi)$
 for each $\xi\in L(\Sigma_\beta)\setminus\{\emptyset\}$.

 Let $n\geq1$ and let $\xi=\xi_0\cdots\xi_{n-1}\in L(\Sigma_\beta)\setminus\{\emptyset\}$. We have  $\Sigma_\beta(0;\xi)\cap K_\beta\neq\emptyset$. 
  Let $\theta_\beta$ denote the Bernoulli measure on $\Sigma_\beta$ associated with the probability vector $(Ma,\frac{1-Ma}{M},\ldots,\frac{1-Ma}{M})$. By  Lemma~\ref{gyak-lem}(b) we have  $\theta_\beta(K_\beta)=1$, and so 
\begin{equation}\label{thm3-eq3}\begin{split}\theta_\beta(\Sigma_\beta(0;\xi)\cap K_\beta)=\theta_\beta&(\Sigma_\beta(0;\xi) )
=(Ma)^{\#\{i\in\{0,\ldots,n-1\}\colon\xi_i=\alpha\}}\\
&\times
\left(\frac{1-Ma}{M}\right)^{\#\{i\in\{0,\ldots,n-1\}\colon\xi_i\in \{\beta_1,\ldots,\beta_M\}\}}.\end{split}\end{equation}
%Let $[\xi]_{\Sigma_D}$ \textcolor{red}{$\phi_\beta^{-1}(\Sigma_\beta(0;\xi))$} denote the set of $\omega\in\Sigma_D$ such that for every $i\in\{0,\ldots,n-1\}$, $\omega_i=\xi_i$ if $\xi_i\neq\alpha$ and $\omega_i\in\{\alpha_1,\ldots,\alpha_M\}$
%if $\xi_i=\alpha$. We have
%\[\textcolor{red}{@=\left\{\omega\in\Sigma_D\colon\omega_i=\xi_i\text{ if }0\leq i\leq n-1,\xi_i\neq\alpha,\omega_i\in\{\alpha_1,\ldots,\alpha_M\}\text{ if }\xi_i=\alpha\right\}}\]
%\begin{equation}\label{thm3-eq1}\phi_\beta([\xi]_{\Sigma_D}\cap B_\beta)=\Sigma_\beta(0,\xi)\cap\phi_\beta(B_\beta).\end{equation}
%where $\eta=\eta_1\cdots\eta_n$ and \[\eta_i=\begin{cases}   \alpha&\text{ if }\xi_i\in\{\alpha_1,\ldots,\alpha_M\},\\ \beta_k&\text{ if }\xi_i=\beta_k,\ k\in\{1,\ldots,M\}\ \text{ for } 1\leq i\leq n. \end{cases}\]
%In what follows we assume $a\in (0,\frac{1}{2M})$. Recall that $g_a=f_{a,\frac{1}{M}-a}$ as in \eqref{ga}.
%By \eqref{l-c-leb} we have $\chi^c(g_a,{\rm Leb})>0$. Write
%$\pi$ for the coding map $\pi_{a,\frac{1}{M}-a}$.
Since $A_\beta\subset B_\beta$, Lemma~\ref{class-lem}(b) gives
 ${\rm Leb}\circ\pi^{-1}(B_\beta)={\rm Leb}\circ\pi^{-1}(A_\beta)=1$. 
We have
\begin{equation}\label{thm3-eq2}\begin{split}{\rm Leb}\circ\pi^{-1}(\phi_\beta^{-1}(\Sigma_\beta(0;\xi)) \cap B_\beta)=&{\rm Leb}\circ\pi^{-1}(\phi_\beta^{-1}(\Sigma_\beta(0;\xi)) )\\
=&(Ma)^{\#\{i\in\{0,\ldots,n-1\}\colon\xi_i=\alpha \}}\\
&\times\left(\frac{1-Ma}{M}\right)^{\#\{i\in\{0,\ldots,n-1\}\colon\xi_i\in \{\beta_1,\ldots,\beta_M\}\}}.\end{split}\end{equation}
Combining \eqref{thm3-eq3}
 and \eqref{thm3-eq2} we get
\[{\rm Leb}\circ\pi^{-1}(\phi_\beta^{-1}(\Sigma_\beta(0;\xi)) \cap B_\beta )=\theta_\beta(\Sigma_\beta(0;\xi)\cap K_\beta ),\]
which yields
  \[\begin{split}{\rm Leb}\circ\pi^{-1}\circ\phi_\beta^{-1}(\Sigma_\beta(0;\xi)\cap K_\beta )&={\rm Leb}\circ\pi^{-1}(\phi_\beta^{-1}(\Sigma_\beta(0;\xi))\cap B_\beta)\\
&=\theta_\beta(\Sigma_\beta(0;\xi)\cap K_\beta ).\end{split}\]
Since $\xi\in L(\Sigma_\beta)\setminus\{\emptyset\}$ is arbitrary, it follows that  
  $(\sigma|_{B_\beta},{\rm Leb}\circ\pi^{-1}|_{B_\beta})$ is isomorphic to $(\sigma_\beta|_{K_\beta},\theta_\beta|_{K_\beta})$ via $\phi_\beta$.
The latter is, by definition, $(Ma,\frac{1-Ma}{M},\ldots,\frac{1-Ma}{M})$-Bernoulli.
Since 
$\pi|_{\pi^{-1}(B_\beta)}$ is injective 
by Lemma~\ref{factor-prop},
the former is isomorphic to $(g_a|_{\pi^{-1}(B_\beta)},{\rm Leb}|_{\pi^{-1}(B_\beta)})$.
Hence
  $(g_a,{\rm Leb})$ is $(Ma,\frac{1-Ma}{M},\ldots,\frac{1-Ma}{M})$-Bernoulli, and so is 
  $(f_a,{\rm Leb})$.\qed

\subsection{Proof of Theorem~\ref{thm-a}}\label{pfthmab-sec}
 Item (a) is contained in \cite[Theorem~1.2]{TY23}.
%we show that ergodic measures whose central Lyapunov exponents are $0$ do not maximize entropy.
%This together with Lemma~\ref{meas-entropy} shows that $\mu_\alpha$ and $\mu_\beta$ are precisely all the ergodic MMEs.
Let $a\in (0,\frac{1}{M})$.
Combining Lemmas~\ref{class-lem}(c),
\ref{ruelle-ineq} and \ref{zero-u}, for any $\mu\in M_0^{\rm e}(\Lambda,f_{\frac{1}{2M},a})$ we have
$h(\mu)\leq\log\sqrt{4M}$. From this and Lemma~\ref{inverse} we get
\begin{equation}\label{pf-thma-eq1}h(\mu)\leq\log\sqrt{4M}\ \text{ for any }\mu\in M_0^{\rm e}(\Lambda,f_{a,\frac{1}{2M}}).\end{equation}
%From \eqref{pf-thma-eq1} and Lemma~\ref{meas-entropy} we obtain $M_{\rm max}^{\rm e}(f_{a,\frac{1}{2M}})=\{\mu_{\alpha,a,\frac{1}{2M}}, \mu_{\beta,a,\frac{1}{2M}}\}$. 
Lemma~\ref{meas-proj} implies that the set $\{\mu\circ p^{-1}\colon\mu\in M_0^{\rm e}([0,1]^3, f_{a,b})\}$ is independent of $b$. Hence, by 
 \eqref{pf-thma-eq1} and Lemma~\ref{meas-proj} again,
%  We apply Lemma~\ref{meas-proj} to obtain $M_{\rm max}^{\rm e}(f_a)\setminus M_0^{\rm e}(f_a)=\{\mu_{\alpha,a,\frac{1}{2M}}^+, \mu_{\beta,a,\frac{1}{2M}}^+\}$, which verifies Theorem~\ref{thm-b}.
for any $b\in (0,\frac{1}{M})$ we have
\begin{equation}\label{pf-thma-eq2}h(\mu)\leq\log\sqrt{4M}<\log(M+1)\ \text{ for any }\mu\in M_0^{\rm e}(\Lambda,f_{a,b }).\end{equation}
From \eqref{pf-thma-eq2} and Lemma~\ref{meas-entropy} we obtain 
$M_{\rm max}^{\rm e}([0,1]^3,f_{a,b})=\{ \mu_{\alpha,a,b},\mu_{\beta,a,b}\}$, which verifies (b).
Item (c) follows from (b) and Lemma~\ref{meas-proj}. 
%We apply Lemma~\ref{meas-proj} to obtain $M_{\rm max}^{\rm e}(f_a)\setminus M_0^{\rm e}(f_a)=\{\mu_{\alpha,a,\frac{1}{2M}}^+, \mu_{\beta,a,\frac{1}{2M}}^+\}$, which verifies Theorem~\ref{thm-b}.

Theorem~\ref{Bernoulli-thm} allows us to compute the entropy %of {\rm Leb} with respect to $F_a$ 
$h(F_a,{\rm Leb})$ as follows:
\begin{equation}h(F_a,{\rm Leb})=\label{entropy-leb}\begin{cases}
-(1-Ma)\log\frac{1-Ma}{M}-Ma\log(Ma)& \text{ for }a\in (0,\frac{1}{2M});\\
-Ma\log a-(1-Ma)\log(1-Ma) & \text{ for }a\in (\frac{1}{2M},\frac{1}{M}).
\end{cases}\end{equation}
By Lemma~\ref{preserve}(b), ${\rm Leb}$ is $f_{a,b}$-invariant if and only if $a+b=\frac{1}{M}$. Hence, to prove Theorem~\ref{thm-a}(d) it suffices to show the following:
\begin{itemize}
\item[(i)] $\mu_{\alpha,a,\frac{1}{M}-a}={\rm Leb}$ if and only if $a=\frac{1}{M+1}$;

\item[(ii)] $\mu_{\beta,a,\frac{1}{M}-a}={\rm Leb}$ if and only if $a=\frac{1}{M(M+1)}$.
\end{itemize}

By %Lemmas~\ref{ruelle-ineq} and \ref{zero-u},
\eqref{pf-thma-eq1}, for any $\mu\in M_0^{\rm e}(\Lambda,f_{\frac{1}{2M},a})$ we have
$h(\mu)<\log(M+1)$. From this inequality and \eqref{entropy-leb} it follows that  
$h(g_a,{\rm Leb})=\log (M+1)$ if and only if $a=\frac{1}{M+1}$ or $a=\frac{1}{M(M+1)}$. 
 In the first case, \eqref{l-c-leb} gives
 $\chi^c(g_a,{\rm Leb})<0$ and \eqref{l-c-mme} gives
$\chi^c(\mu_{\alpha,a,\frac{1}{M}-a})<0$, and so
 we obtain $\mu_{\alpha,a,\frac{1}{M}-a}={\rm Leb}$. 
 In the second case, the same reasoning yields
$\mu_{\beta,a,\frac{1}{M}-a}={\rm Leb}$. Hence (i) and (ii) hold.\qed

\subsection{Proof of Theorem~\ref{thm-c}}\label{pfthmc-sec}
 To prove (a), 
it suffices to show that for any
 $(a_*,b_*)\in \Delta$ and
for any continuous function $\varphi\colon[0,1]^3\to\mathbb R$, the map
$(a,b)\in \Delta\mapsto\int \varphi{\rm d}\mu_{\gamma,a,b}$
is continuous at $(a_*,b_*)$
for each $\gamma\in\{\alpha,\beta\}$. 
A proof of this amounts to estimating for 
$\nu_\gamma$-almost every $\omega\in A_\gamma$ the difference between 
$\pi_{a_*,b_*}^{-1}(\omega)$ and $\pi_{a,b}^{-1}(\omega)$  for all $(a,b)\in\Delta$ that is sufficiently close to $(a_*,b_*)$.

Let
$\varphi\colon[0,1]^3\to\mathbb R$ be  continuous and
let $\varepsilon>0.$
Pick $m_0>0$ such that 
\begin{equation}\label{thmc-eq0}\sup
\left\{|\varphi(x)-\varphi(y)|\colon x,y\in[0,1]^3,{\rm dist}(x,y)< 8(M+1)^{-m_0}\right\}<\varepsilon.\end{equation}
For $m$, $n\in\mathbb N_0$, define two sets
\[G_{\alpha,m,n}=\left\{\omega\in\Sigma_D\colon
\lim_{i\to\infty }H_i(\omega)= \infty\ \text{ and } \ \sup_{i\leq -n}H_i(\omega)\leq -m\right\},\]
\[G_{\beta,m,n}=\left\{\omega\in\Sigma_D\colon\sup_{i\geq n }H_i(\omega)\leq -m\ \text{ and } \ \lim_{i\to-\infty}H_i(\omega)=\infty\right\}.\]
Let $\gamma\in\{\alpha,\beta\}$.
For each fixed $m$, $G_{\gamma,m,n}$
 increases as $n$ increases.
For each fixed $n$, $G_{\gamma,m,n}$
decreases as $m$ increases.
Moreover, $G_{\gamma,m,n}$ is empty if $n<m$.
Note that 
\[A_\gamma=\bigcap_{m=0}^\infty\bigcup_{n=m}^\infty G_{\gamma,m,n}.\]
We have
 $\nu_\gamma(A_\gamma)=1$ as in \eqref{measure-nu}, and 
$\bigcup_{n=m}^\infty G_{\gamma,m,n}$
 decreases as $m$ increases.
 %$\bigcap_{N_0\geq0}\bigcup_{N\geq N_0}G_{K,N}$ is increasing in $K$,
Hence, we can pick integers $m_1> m_0$ and $N> m_1$ such that
 \begin{equation}\label{thmc-eq4}\max\{a_*,1-Ma_*,b_*,1-Mb_*\} <(M+1)^{-\frac{m_0}{m_1}},\end{equation}
 and
 \begin{equation}\label{thmc-eq5}\nu_\gamma\left(\bigcup_{n=m_1}^{N}G_{\gamma,m_1,n}\right)=\nu_\gamma(G_{\gamma,m_1,N})>1-\varepsilon.\end{equation}

To proceed we need some notation. For $x\in\mathbb R^3$ and $\delta>0$, let $R(x;\delta)$ denote the closed cube in $\mathbb R^3$ centered at $x$ all whose edges are parallel to the coordinate axes and have the Euclidean length $\delta$.
By Lemma~\ref{factor-prop}, $\pi_{a,b}^{-1}(\omega)$ is a singleton 
for $\nu_\gamma$-almost every $\omega\in A_\gamma$.
Let us still denote by $\pi_{a,b}^{-1}(\omega)$ the point in this singleton despite the possible ambiguity. 

Let $(a,b)\in\Delta$. 
For each $\omega\in G_{\gamma,m_1,N}$
consider the set 
\begin{equation}\label{cuboid}P_{a,b}(\omega)=\bigcap_{i=-N}^{N-1} f_{a,b}^{-i}(\Omega_{\omega_i}),\end{equation}
which has non-empty interior.
Since $N$ is finite,
%for each $(a,b)\in\Delta$ 
the total number of sets of the form \eqref{cuboid} is finite.
%We estimate the size of each cuboid from above.
As in the proof of Lemma~\ref{factor-prop}, for $t=u,c,s$
and two points $x=(x_u,x_c,x_s)$, $y=(y_u,y_c,y_s)$ in $[0,1]^3$,
  write
$|x-y|_t=|x_t-y_t|$.
For a subset $S$ of $[0,1]^3$ 
write ${\rm diam}^t(S)=\sup_{x,y\in S} |x-y|_t$. 
Recall that the $x_u$-direction is expanding by factor $\frac{1}{a}$ or $\frac{1}{1-Ma}$, and the $x_s$-direction is contracting by factor $1-Mb$ or $b$ under the forward iteration of $f_{a,b}$. For all $\omega\in G_{\gamma,m_1,N}$  the following estimates hold:
\[\begin{split}&{\rm diam}^u(P_{a,b}(\omega ))\leq\max\{a,1-Ma\}^{m_1};\\
&{\rm diam}^c(P_{a,b}(\omega ))\leq (M+1)^{-m_1};\\
&{\rm diam}^s(P_{a,b}(\omega ))\leq\max\{b,1-Mb\}^{m_1}.\end{split}\]
Combining these three estimates for $(a,b)=(a_*,b_*)$ with \eqref{thmc-eq4},
 we have
\[P_{a_*,b_*}(\omega )\subset {\rm int }(R(\pi_{a_*,b_*}^{-1}(\omega);2(M+1)^{-m_0})),\]
for $\nu_\gamma$-almost every $\omega\in G_{\gamma,m_1,N}$.
Since $N$ is finite and the set $\bigcup_{i=-N}^{N-1}f_{a,b}^i(\bigcup_{\gamma'\in D}\partial\Omega_{\gamma'})$ depends smoothly on the parameter $(a,b)$, there exists an open neighborhood $U$ of $(a_*,b_*)$ in $\Delta$ such that 
  we have
 \begin{equation}\label{thmc-eq7}P_{a,b}(\omega )\subset R(\pi_{a,b}^{-1}(\omega);2(M+1)^{-m_0})\ \text{  and } \  
P_{a_*,b_*}(\omega )\cap P_{a,b}(\omega)\neq\emptyset,\end{equation}
for all $(a,b)\in U$ and $\nu_\gamma$-almost every $\omega\in G_{\gamma,m_1,N}$.
From \eqref{thmc-eq7} 
 we get
\[\pi_{a,b}^{-1}(\omega)\in P_{a,b}(\omega)\subset R(\pi_{a_*,b_*}^{-1}(\omega);
2(M+1)^{-m_0}),\]
which yields the key estimate
\begin{equation}\label{thmc-eq3}{\rm dist}(\pi_{a,b}^{-1}(\omega),\pi_{a_*,b_*}^{-1}(\omega))< 8(M+1)^{-m_0},\end{equation}
 for all $(a,b)\in U$ and $\nu_\gamma$-almost every $\omega\in G_{\gamma,m_1,N}$.

To finish, let $(a,b)\in U$.
From \eqref{thmc-eq0} and \eqref{thmc-eq3}, for $\nu_\gamma$-almost every $\omega\in G_{\gamma,m_1,N}$
we have
$|\varphi(\pi_{a,b}^{-1}(\omega))-\varphi(\pi_{a_*,b_*}^{-1}(\omega))|<\varepsilon$. Therefore
\[\begin{split}\left|\int \varphi {\rm d}\mu_{\gamma,a,b}-\int \varphi {\rm d}\mu_{\gamma,a_*,b_*}\right|=&\left|\int  \varphi\circ\pi_{a,b}^{-1}{\rm d}\nu_\gamma-\int  \varphi\circ\pi_{a_*,b_*}^{-1}{\rm d}\nu_\gamma\right|\\
\leq&\int_{G_{\gamma,m_1,N}}  |\varphi\circ\pi_{a,b}^{-1}-  \varphi\circ\pi_{a_*,b_*}^{-1}|{\rm d}\nu_\gamma\\
&+\int_{ \Sigma_D\setminus G_{\gamma,m_1,N} }  |\varphi\circ\pi_{a,b}^{-1}-  \varphi\circ\pi_{a_*,b_*}^{-1}|{\rm d}\nu_\gamma \\
\leq&\varepsilon\nu_\gamma(G_{\gamma,m_1,
N})+\nu_\gamma(\Sigma_D\setminus G_{\gamma,m_1,N})\|\varphi\|\\
\leq&\varepsilon+\varepsilon\|\varphi\|,\end{split}\]
where $\|\varphi\|=\max_{x\in[0,1]^3}|\varphi(x)|.$
To deduce the last inequality we have used \eqref{thmc-eq5}.
Since %$f$ is an arbitrary continuous function and 
$\varepsilon>0$ is arbitrary, it follows that
$(a,b)\in \Delta\mapsto\int \varphi{\rm d}\mu_{\gamma,a,b}$
is continuous at $(a_*,b_*)$, which
 verifies (a).
 Item (b) follows from (a) and Lemma~\ref{meas-proj}.
 %The proof of Theorem~\ref{thm-c} is complete.
 \qed

 \appendix
\def\thesection{\Alph{section}}
\section{Some variants of the heterochaos baker maps}\label{app-A}
It is natural to ask to what extent our results can be generalized to piecewise affine maps that are  variants of the heterochaos baker maps. As one of the simplest variants,
for an integer $N\geq2$ with $M\neq N$
and $a\in(0,\frac{1}{M})$
consider a piecewise affine map $f_{(a,N)}\colon[0,1]^2\to[0,1]^2$ defined as follows:
(i) $f_{(a,N)}|_{[0,Ma)\times[0,1]}=f_a|_{[0,Ma)\times[0,1]}$;
(ii) Split $[Ma,1]\times[0,1]$ into $N$ rectangular domains appropriately with the $x_c$-diameter $\frac{1}{N}$. Then $f_{(a,N)}$ maps each of these domains to cover $(0,1)^2$ without flip and rotation.
The map $f_{(a,3)}$ with $M=2$ is shown in
\textsc{Figure}~\ref{homburg}. 

For $f_{(a,N)}$,
extending the argument in \cite{STY21} one can show that the set of $1$-unstable periodic points and that of $2$-unstable periodic points are dense in $[0,1]^2$. 
We suspect that the left shift acting on
the coding space of $f_{(a,N)}$ determined by the natural partition of $[0,1]^2$ is intrinsically ergodic, and so is $f_{(a,N)}$:
 the dynamics of $f_{(a,N)}$ is heterochaotic on topological level and not so on ergodic level.
%The Lebesgue measure is not $f_{a,N}$-invariant.
Homburg and Kalle \cite{HK22} considered
iterated function systems intimately related to the dynamics of $f_{(a,N)}$ in the $x_c$-coordinate.

%Homburg and Kalle \cite{HK22} considered iterated function systems on the interval $[0,1]$ generated by expanding maps and contracting ones. In other words, they considered the projection to the $x_c$-axis. 

\begin{figure}
\begin{center}
\includegraphics[height=4.4cm,width=12cm]{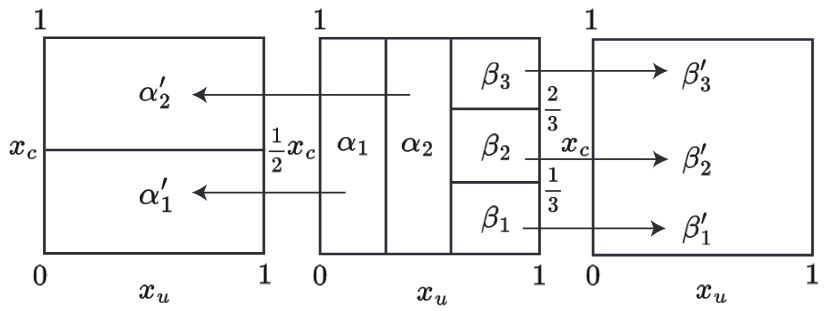}
%{heterochaos2d-ii.eps}
\caption
{The piecewise affine map $f_{(a,3)}\colon[0,1]^2\to[0,1]^2$ with $M=2$. }\label{homburg}
\end{center}
\end{figure}

\section{Cycles arising from heterochaos horseshoe maps}
Piecewise smooth dynamical systems may be models of 
real world phenomena in time evolution, 
and real significance may be accorded only to those models or properties that are robust under perturbations of systems. 
Clearly the heterochaos baker maps are not robust.

 One can deform the baker-like map as in \textsc{Figure}~\ref{horseshoe} into a structurally stable horseshoe map.
    For the heterochaos baker map on $[0,1]^3$, 
   we can perform a similar deformation to obtain a diffeomorphism $\phi$ that robustly exhibits the unstable dimension variability \cite{STY22}.
Moreover, $\phi$ has a heterodimensional cycle \cite{bonatti_1996,BonDia08,D95} that has some robustness under small perturbations of the system.
%, a robust organization of invariant manifolds involving non-transverse intersections.
   In this appendix we recall the definition of $\phi$ in \cite{STY22}, and
    comment on the existence of robust cycles 
    with a comparison to earlier different developments.

\begin{figure}
\centering
\includegraphics[width=.6\textwidth]{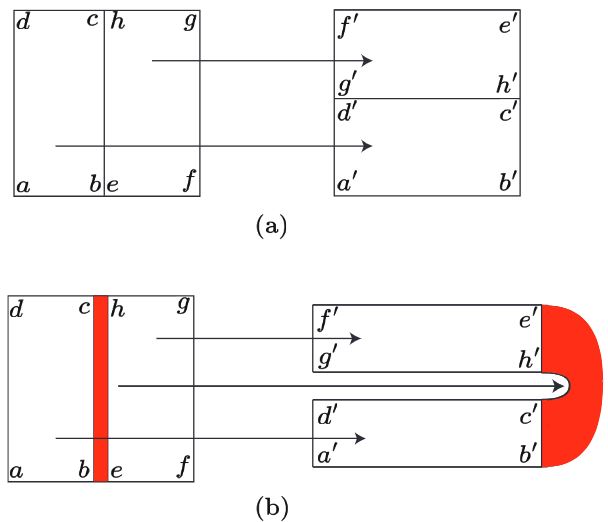}
%{horseshoe-new.eps}
\caption{(a) A baker-like map 
 from $[0,1]^2$ into itself. (b) An affine horseshoe map on $[0,1]^2$ to $\mathbb R^2$. The vertex $a$ is mapped to $a'$, and so on. Squeezing the folded part recovers the baker-like map.}\label{horseshoe}
\end{figure}

A {\it block} is a Cartesian product of three non-degenerate compact intervals in $\mathbb R^3$ all of whose sides are parallel to one of the axes of coordinates of $\mathbb R^3$.
Given two blocks $X$ and $Y$, we say $X$ {\it stretches across} $Y$
if $X$ does not intersect the edges of $Y$ and $X\setminus Y$ has two connected components.
With a slight abuse of notation,
 Let $\Omega_{\alpha_1}$, $\Omega_{\alpha_2}$, $\Omega_{\beta_1}$, $\Omega_{\beta_2}$ be pairwise disjoint blocks
in $[0,1]^3$ such that
 $\Omega_{\alpha_2}$ is a translate of $\Omega_{\alpha_1}$ in the $x_u$-direction and
 $\Omega_{\beta_2}$ is a translate of $\Omega_{\beta_1}$
 in the $x_c$-direction. Similarly,
 let $\Omega_{\alpha_1^\prime}$, $\Omega_{\alpha_2^\prime}$, $\Omega_{\beta_1^\prime}$, $\Omega_{\beta_2^\prime}$
be pairwise disjoint 
blocks in $[0,1]^3$
such that $\Omega_{\alpha_2^\prime}$ is a translate of $\Omega_{\alpha_1^\prime}$ in the $x_c$-direction and
 $\Omega_{\beta_2^\prime}$ is a translate of $\Omega_{\beta_1^\prime}$
 in the $x_s$-direction.
We assume:

\begin{itemize}
 \item[(i)] $\Omega_{\alpha_1^\prime}$  stretches across $\Omega_{\alpha_1}$ and $\Omega_{\alpha_2}$, and $\Omega_{\alpha_2^\prime}$ stretches across $\Omega_{\alpha_1}$ and $\Omega_{\alpha_2}$.

\item[(ii)] $\Omega_{\beta_1}$ stretches across $\Omega_{\beta_1^\prime}$ and $\Omega_{\beta_2^\prime}$, and $\Omega_{\beta_2}$ stretches across $\Omega_{\beta_1^\prime}$ and $\Omega_{\beta_2^\prime}$.

\item[(iii)] 
$\Omega_{\alpha_1^\prime}$ stretches across $\Omega_{\beta_1}$, and $\Omega_{\alpha_2^\prime}$ stretches across $\Omega_{\beta_2}$.

\item[(iv)] $\Omega_{\alpha_1}$ stretches across $\Omega_{\beta_1^\prime}$ and $\Omega_{\beta_2^\prime}$, and
$\Omega_{\alpha_2}$ stretches across $\Omega_{\beta_1^\prime}$ and $\Omega_{\beta_2^\prime}$.
\end{itemize}
\noindent See \textsc{Figure}~\ref{fig:coupledhorseshoe}(a), (b), (c), (c) respectively. 
%Set \[R_1=\alpha_1\cup \alpha_2\quad\text{and}\quad R_2= \beta_1\cup \beta_2.\]
%Set \[R_1=A\cup D\quad\text{and}\quad R_2= B^*\cup C^*.\]

% \begin{itemize}
%  \item[(i)] $\alpha_1^*$  stretches across $\alpha_1$ and $\alpha_2$, and $\alpha_2^*$ stretches across $\alpha_1$ and $\alpha_2$.

% \item[(ii)] $\beta_1$ stretches across $\beta_1^*$ and $\beta_2^*$, and $\beta_2$ stretches across $\beta_1^*$ and $\beta_2^*$.

% \item[(iii)] 
% $\alpha_1^*$ stretches across $\beta_1$, and $\alpha_2^*$ stretches across $\beta_2$.

% \item[(iv)] $\alpha_1$ stretches across $\beta_1^*$ and $\beta_2^*$, and
% $\alpha_2$ stretches across $\beta_1^*$ and $\beta_2^*$.
% \end{itemize}
% \noindent See FIGURE~\ref{fig:coupledhorseshoe}(a), (b), (c), (c) respectively. 
% Set \[R_1=A\cup D\quad\text{and}\quad R_2= B^*\cup C^*.\]

%For an integer $r\geq1$,  
%let ${\rm Diff}^r(\mathbb R^3)$ denote the space of $C^r$ diffeomorphisms of $\mathbb R^3$
%endowed with the $C^r$ compact open topology.
Let $\phi$ be a $C^1$ diffeomorphism on $\mathbb R^3$ onto its image
that maps $\Omega_{\alpha_1}$, $\Omega_{\alpha_2}$, $\Omega_{\beta_1}$, $\Omega_{\beta_2}$ affinely to $\Omega_{\alpha_1^\prime}$, $\Omega_{\alpha_2^\prime}$, $\Omega_{\beta_1^\prime}$, $\Omega_{\beta_2^\prime}$ respectively
with diagonal Jacobian matrices, see \textsc{Figure}~\ref{fig:coupledhorseshoe}(d).
Condition
(i) implies that the restriction
$\phi|_{\Omega_{\alpha_1}\cup \Omega_{\alpha_2}}$ is a horseshoe map
whose unstable dimension is $1$. Condition
 (ii) implies that
$\phi|_{\Omega_{\beta_1}\cup \Omega_{\beta_2}}$
is a horseshoe map  whose unstable dimension is $2$. 
Conditions (iii) (iv) determine how these two horseshoe maps are coupled.
Let us call $\phi$ a 
{\it heterochaos horseshoe map}. 

\begin{figure}
\centering
\includegraphics[width=.9\textwidth]{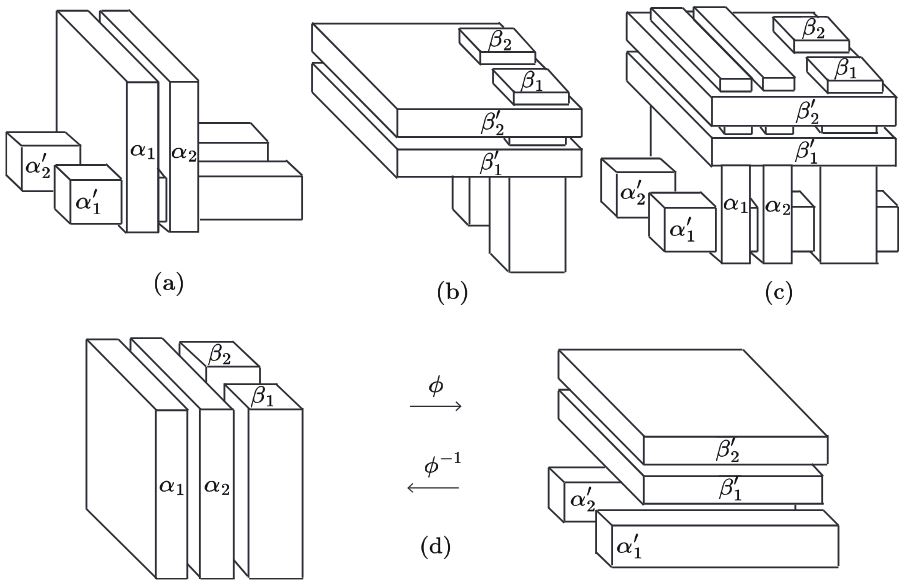}
%{heterochaos-ent3d-5.eps}
\caption{The heterochaos horseshoe map $\phi$. Each of the four blocks labeled with $\gamma\in\{\alpha_1,\alpha_2,\beta_1,\beta_2\}$ is mapped by $\phi$ to the block labeled with $\gamma'$.}
\label{fig:coupledhorseshoe}
\end{figure}

If $\varphi$ is another diffeomorphism that is sufficiently $C^1$ close to $\phi$,
the compact $\varphi$-invariant set
\[\Lambda=\bigcap_{n=-\infty}^{\infty} \varphi^{-n}(\Omega_{\alpha_1}\cup \Omega_{\alpha_2}\cup \Omega_{\beta_1}\cup \Omega_{\beta_2})\]
contains two closed $\varphi$-invariant sets
$\Gamma=\Gamma(\varphi)$ and $\Sigma=\Sigma(\varphi)$ given by
\[\Gamma=\bigcap_{n=-\infty}^{\infty} \varphi^{-n}(\Omega_{\alpha_1}\cup \Omega_{\alpha_2})\quad\text{and}\quad 
\Sigma=\bigcap_{n=-\infty}^{\infty} \varphi^{-n}(\Omega_{\beta_1}\cup \Omega_{\beta_2}),\]
which are hyperbolic sets of unstable dimensions $1$ and $2$ respectively.
Extending the arguments in \cite{STY21} one can show the transitivity, as well as the density of hyperbolic periodic points in $\Lambda$ with different unstable dimensions.
 If $\varphi$ is $C^2$, the Lebesgue measure of $\Lambda$ is zero.

Let $W^u(\Gamma)$ (resp. $W^s(\Sigma)$)
denote the union of the unstable (resp. stable) manifolds of points in $\Gamma$ (resp. $\Sigma)$, namely, 
\[\begin{split}W^u(\Gamma)&=\{x\in\mathbb R^3\colon \exists y\in\Gamma \text{ s.t. } {\rm dist}(\varphi^n(x),\varphi^n(y))\to0\text{ as }n\to-\infty\},\\
W^s(\Sigma)&=\{x\in\mathbb R^3\colon \exists y\in\Sigma \text{ s.t. } {\rm dist}(\varphi^n(x),\varphi^n(y))\to0\text{ as }n\to\infty\},\end{split}\]
where ${\rm dist}(\cdot,\cdot)$ denotes the Euclidean distance on $\mathbb R^3$.
 % The sets $W^u(\Gamma)$ and $W^s(\Sigma)$ are contained in locally invariant $C^2$ surfaces.
  These sets
 intersect each other as depicted in \textsc{Figure}~\ref{fig:heteroclinic-intersection}. All their intersections are quasi-transverse, while
  $W^s(\Gamma)$ intersects $W^u(\Sigma)$  transversely. Therefore, the two hyperbolic sets $\Gamma$ and $\Sigma$ are cyclically related.

\begin{figure}
\centering
\includegraphics[height=0.5\linewidth,width=0.5\linewidth]{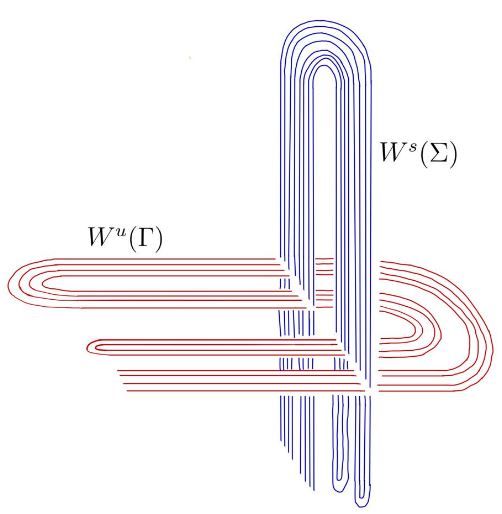}
%{heteroclinic-shishikura60-12-2.eps}
%{paperII-figs/paper2-figfontnew3.JPG}
\caption{The sets $W^u(\Gamma)$, $W^s(\Sigma)$ for $\varphi$ $C^2$-close to $\phi$ and their fractal intersection.
 The sets $W^u(\Gamma)$ and $W^s(\Sigma)$ are contained in locally invariant $C^2$ surfaces}.
\label{fig:heteroclinic-intersection}
\end{figure}

 We say a diffeomorphism $\psi$ on $\mathbb R^3$
 %\in{\rm Diff}^r(\mathbb R^3)$ $(r\geq1)$ 
 has a {\it heterodimensional cycle} associated to its compact transitive hyperbolic sets $\Psi$ and $\Upsilon$ if the unstable dimensions of  $\Psi$ and $\Upsilon$ are different, and
 $W^s(\Psi)\cap W^u(\Upsilon)\neq\emptyset$ and $W^s(\Upsilon)\cap W^u(\Psi)\neq\emptyset$.
 Since non-transverse intersections in heterodimensional cycles
 can easily be destroyed by perturbations of diffeomorphisms, the following notion is useful.
A heterodimensional cycle of $\psi$ associated to transitive hyperbolic sets $\Psi$ and $\Upsilon$ is {\it $C^r$-robust} \cite{BonDia08}
if there is a $C^r$ neighborhood $\mathcal V$ of $\psi$
such that any diffeomorphism in $\mathcal V$
has a heterodimensional cycle associated to the 
continuations of $\Psi$ and $\Upsilon$.
By Kupka-Smale's theorem, if $\psi$ has a $C^r$-robust heterodimensional cycle associated to $\Psi$ and $\Upsilon$, either $\Psi$ or $\Upsilon$ is a non-trivial hyperbolic set.

Robust heterodimensional cycles are often constructed by perturbations. It may be useful to exhibit a concrete example of a diffeomorphism
that has a robust heterodimensional cycle. 
 If $\phi$ satisfies additional hypotheses, 
  then it has a $C^2$-robust heterodimensional cycle
 associated to $\Gamma$ and $\Sigma$. %One can estimate the Hausdorff dimension of $W^u(\Gamma)\cap W^s(\Sigma)$.
For details, see \cite{STY22}.

%\textcolor{cyan}{Copy from \cite{STY22} intro}
%\subsection{Comparison with blender approach}
There is an earlier, different line  of research where the intersection analysis of invariant manifolds are
carried out using {\it blenders}.
For a comparison,
among the many types of blenders 
the one in \cite{ACW21} has an advantage
in that an affine
horseshoe and its perturbations were also considered.
For other definitions of blenders, see e.g., \cite{BCDW16,bonatti_1996} and 
\cite[Section~6.2]{BDV04}.
Key properties of a $d_{cs}$-stable (resp. $d_{cu}$-unstable) blender, in dimension three for example, is that a related one-dimensional local  stable (resp. unstable) manifolds form
a `topological surface' which intersects any curve `transverse' to it, and that this property is $C^1$-robust. 
These key properties are used to blend $C^1$-robust intersections between one-dimensional invariant manifolds.
For a diffeomorphism $\varphi$ on $\mathbb R^3$ that is sufficiently $C^2$-close to $\phi$, 
 the local stable manifolds of
 $\Sigma$ 
 are contained in a $C^2$ surface
\cite[Proposition~3.3]{STY22}.
 Hence, $\Sigma$ is not a $d_{cs}$-stable blender: the first key property obviously fails.
 Similarly,
 $\Gamma$ is not a $d_{cu}$-unstable blender since its local unstable manifolds 
 are contained in a $C^2$ surface.

 For heterochaos horseshoe maps
 satisfying some additional assumptions,
 using blenders
  one can construct $C^1$-robust heterodimensional cycles associated to hyperbolic sets other than $\Gamma$, $\Sigma$.
  For details, see \cite[Appendix]{STY22}.

\subsection*{Acknowledgments}
We thank the referee for his/her careful reading of the manuscript and giving valuable suggestions.
We thank Shigeki Akiyama, 
Masayuki Asaoka, Hiroshi Kokubu, 
Yushi Nakano, 
Mitsuhiro Shishikura, Toshi Sugiyama, Naoya Sumi, Masato Tsujii for fruitful discussions.
Part of this work have been done while YS and HT were visiting University of Maryland, HT and KY were visiting University of Tsukuba, and JAY was visiting Keio University for the international conference `KiPAS Dynamics Days 2023' (August 27-31, 2023). YS was supported by the JSPS KAKENHI 19KK0067, 23H04465, 24K00537 and the JSPS Bilateral Open Partnership Joint Research Projects JPJSBP 120229913.
HT was supported by the JSPS KAKENHI 19K21835, 19KK0067, 23K20220 and the Keio Gijuku Academic Development Funds. 
KY was supported by the JSPS KAKENHI
21K03321.
 This work was supported by Research Institute for Mathematical Sciences,
an International Joint Usage/Research Center located in Kyoto University.

\subsection*{Declarations}\

\

{\bf Conflict of interest} \ The authors declare no conflict of interest.

 \bibliographystyle{amsplain}

\end{document}